\documentclass[times]{amsart}
\usepackage{amssymb, hyperref, fullpage}

\theoremstyle{plain}
\newtheorem{proposition}[equation]{Proposition}
\newtheorem{remarks}[equation]{Remarks}
\newtheorem{theorem}[equation]{Theorem}
\newtheorem{definition}[equation]{Definition}
\newtheorem{corollary}[equation]{Corollary}
\newtheorem{lemma}[equation]{Lemma}
\newtheorem{remark}[equation]{Remark}
 \newcommand{\esm}{\end{smallmatrix}\right)}
\newcommand{\bpm}{\begin{pmatrix}}
\newcommand{\ebpm}	{\end{pmatrix}}
\newcommand{\bspm}{\left(\begin{smallmatrix}}
\newcommand{\espm}	{\end{smallmatrix}\right)}
\newcommand{\R}{\mathbb R}
\newcommand{\A}{\mathbb A}
\newcommand{\Z}{\mathbb Z}
\newcommand{\bc}{\mathbb C}
\newcommand{\m}		{\hphantom{-}}

\newcommand{\fa}		{\mathfrak {a}}
\newcommand{\gm}		{\gamma}
\renewcommand{\i}{\infty}

\newcommand{\Q}		{\mathbb {Q}}
\newcommand{\C} {\mathbb C}

\newcommand{\ga}{\bc [GL(2, \Q)^+]}
\newcommand{\gltqp}  {GL(2, \Q)^+}
\newcommand{\fm}{\mathfrak{m}}
\renewcommand{\P}{\mathbb P}
\newcommand{\bs} {\backslash}

\newcommand{\modi} {\pmod{\mathcal I}}
\newcommand{\uhp}{\mathfrak h}
\newcommand{\fb}{\mathfrak b}

\begin{document}

\title{
Fourier expansions of GL(2) newforms  at various cusps
}

\author{Dorian Goldfeld, Joseph Hundley, and Min Lee }

\address{
Mathematics Department,       Columbia University,
       New York, NY 1002,
and
Department of Mathematics,
Mailcode 4408,
Southern Illinois University,
1245 Lincoln Drive,
Carbondale, IL 62901
}

\thanks{The first author was partially supported by  NSF grant 1001036. The second author was   supported by NSA grant MSPF-08Y-172}

\begin{abstract}
This paper studies the Fourier expansion  of Hecke-Maass eigenforms for $GL(2, \mathbb Q)$  of arbitrary weight, level, and character at various cusps. Translating well known results in the theory of adelic automorphic representations into classical language, a multiplicative expression for the Fourier coefficients at any cusp is derived.  In general, this expression involves Fourier coefficients at several different cusps.  A sufficient condition for the existence of multiplicative relations among Fourier coefficients at a single cusp is given.  It is shown that if the level is 4 times (or in some cases 8 times) an odd squarefree number then  there are multiplicative relations at every cusp. We also show that a local representation of $GL(2, \mathbb Q_p)$ which is isomorphic to a local factor of a global cuspidal automorphic representation generated by the adelic lift of a newform of arbitrary weight, level $N$, and  character $\chi\pmod{N}$ cannot be supercuspidal if $\chi$ is primitive. Furthermore, it is supercuspidal if and only if at every cusp (of width $m$ and cusp parameter = 0)  the $mp^\ell$ Fourier coefficient, at that cusp, vanishes for all sufficiently large positive integers $\ell$.  In the last part of this paper a three term identity involving the Fourier expansion at three different cusps is derived. 
\end{abstract}

\maketitle
\tableofcontents
\section{Introduction}\label{s:Intro}

This paper was inspired by the following question raised by Harold Stark.
  \vskip 10pt
  \centerline{\it When do the Fourier coefficients at a cusp of a classical newform for $GL(2)$ satisfy multiplicative relations?}
  \vskip 10pt
  
   By a celebrated theorem of Hecke the Fourier coefficients at the cusp $\infty$ satisfy multiplicative relations.  Similar results regarding the Fourier coefficients at an arbitrary cusp, for a newform on $\Gamma_0(N)$, have been proven by Asai when the level $N$ is squarefree \cite{Asai:1976}.
 The proof is, roughly, that if $N$ is squarefree, then the cusps are represented by the 
quotients $1/t$ with $t$ running over the positive divisors of $N$, and the 
corresponding Fricke involution $W_t$ maps the cusp $\infty$ to $1/t$ and, at the 
same time, acts as an involution on the space of newforms commuting with all 
Hecke operators. 
  Kojima was able to obtain a similar result by another method in the case when the level is  4q where q is a prime \cite{Kojima:1979}. 
  Further, Asai's reasoning may be applied to general $N,$ yielding multiplicative relations at any cusp which is related to $\infty$ by a Fricke involution.  (Such cusps are in one-to-one correspondence with divisors $d$ of $N$ such that gcd$(d, N/d)=1$.)
  In the general case, very little was known about the Fourier coefficients at cusps which can not be mapped to  $\infty$ by a Fricke involution,
  at least in classical terms.

Many of the results of this paper are probably well known to experts in the theory of adelic automorphic representations, but the translation back to classical language is not so easy. We have been unable to find anything on this topic in the literature (beyond the aforementioned results of Asai and Kojima) so we thought it would be useful to write up our explicit  results.

The celebrated tensor product theorem \cite{Flath:1979} (see also \cite{Bump:1997} \S3.4) states  that an  irreducible adelic automorphic representation factors as a tensor product of local representations. Our theorem \ref{t:3.8} may be viewed as a translation of the tensor product theorem into a statement about Fourier coefficients of Maass forms. 
In theorem \ref{t:main}, we discuss 
the special case of prime power level
in greater detail.   These theorems give a partial answer to Stark's question, by providing a multiplicative expression for the Fourier coefficients at an arbitrary cusp.  However, in most cases this expression will involve Fourier coefficients at several other cusps.   Theorem  \ref{t:suffCond} gives a sufficient condition for multiplicative relations involving Fourier coefficients at a single cusp. As a consequence, we deduce that there are multiplicative relations at every cusp if the level $N$ is equal to 4 times a squarefree odd number. Further, if $N$ is 8 times a squarefree odd number, then theorem \ref{t:suffCond} will imply multiplicative relations at all cusps except for those of the form $a/b$ with $2 || b.$ In proposition \ref{p:eight} we consider such cusps in detail.

 When the level of a Maass form is divisible by 
 the square of a prime, one expects more complicated behavior, both classically and representation-theoretically.  
 From the representation-theoretic point 
 of view, this added complexity is in part
 due to the existence of 
  supercuspidal representations.   In Theorem~\ref{t:3.9},  we unpack the classical 
content of a criterion for a representation of $GL(2, \mathbb Q_p)$ to be supercuspidal.
This  states that an irreducible 
smooth representation of $GL(2, \mathbb Q_p)$ is
supercuspidal if and only if its Kirillov model is the 
Schwartz space of $\Q_p^\times$ 
(\cite{Jacquet:1970}, Proposition 2.16, \cite{Godement:1970}, Theorem 3).
As we show in Theorem~\ref{t:3.9}, this implies that a  representation of $GL(2, \mathbb Q_p)$ which is isomorphic to the local representation factor of an irreducible automorphic representation of $GL(2, \mathbb A)$ (generated by the adelic lift of a classical Hecke-Maass newform) is supercuspidal if and only if  the $m_\mathfrak a p^\ell$ Fourier coefficient vanishes for all sufficiently large integers $\ell$ for each cusp $\mathfrak a \in \mathbb Q\cup\{\infty\}$. Here $m_\mathfrak a$, also called the width of the cusp,  is an  integer given by \eqref{e:sigma_a}. It is  well-known
(see \cite{Bernstein:1984}, \cite{Gelbart:1996} and also Proposition 5.1 of \cite{Shahidi:1990})
 that every supercuspidal representation of $GL(2, \mathbb Q_p)$ can be realized as a component of a global  irreducible cuspidal automorphic representation of $GL(2, \mathbb A), $ and that 
 every such representation may be  associated to a normalized Hecke-Maass newform.  
 Theorem~\ref{t:3.9} 
gives a criterion for recognizing
 those Hecke-Maass newforms which are connected with supercuspidals in this way.   
 By combining this with a result from the 
 classical side, we deduce, in 
  Corollary~\ref{c:3.10} 
that a local representation of $GL(2, \mathbb Q_p)$  cannot be supercuspidal if it isomorphic to a local factor of a global cuspidal automorphic representation generated by the adelic lift of a Hecke-Maass newform of weight $k$, level $N$ and character 
$\chi\hskip -3pt\pmod{N}$ if $\chi$ is a primitive character.  This shows that supercuspidal representations can only arise from Hecke-Maass newforms of level $N$ with imprimitive characters $\hskip -5pt\pmod{N}$.  In this connection, we would also like to mention a result of Casselman (see \cite{Casselman:1973}), which implies that 
supercuspidal representations of $GL(2, \mathbb Q_p)$
 can only arise from Masss-Hecke newforms of level $N$  such that $p^2\mid N.$   For more detail see the 
 remark which follows the proof of Corollary~\ref{c:3.10}.
 
 Theorem~\ref{t:3.8} is a classical statement.  It 
 is natural to ask whether it is possible to prove it 
 classically.  
This problem is considered in 
section~\ref{s:TowardsClassicalProof}.
Although we do not recover the full strength of 
theorem~\ref{t:3.8}, we do obtain   relations between the Fourier coefficients at different cusps for Maass forms of arbitrary weight level and character, using a purely classical approach adapted from that employed by Kojima in the case $N=4q.$ The main result is given in Theorem~\ref{t:ThreeTerm} where a three term relation (involving three different cusps) is obtained.

\vskip 10pt\noindent
{\bf Acknowledgements:} The authors would like to thank Herv\'e Jacquet,  Muthu Krishnamurthy, Omer Offen, Scott Ahlgren, and Jeremy Rouse for some helpful conversations.
\section{Newforms and Hecke operators for $\Gamma_0(N)$}\label{s:NewAndHecke}

In this section we briefly review the theory of newforms and Hecke operators for the group $\Gamma_0(N)$. We fix an integer $k$ (called the weight), an integer $N \ge 1$ (called the level), and a Dirichlet character $\chi: \left(\mathbb Z/N\mathbb Z\right)^\times \to \mathbb C^\times$. Let $\mathfrak h$ denote the upper half plane. For any function $f:\mathfrak h\to\mathbb C$,
and any matrix $\gamma \in GL(2, \mathbb R)$ of positive determinant, 
 define the slash operator
 \begin{equation}\label{e:slashOperator}
 \left(f \big |_k \gamma\right)(z) :=
  \left(\frac{cz+d}{|cz+d|}\right)^{-k} f\left(\frac{az+b}{cz+d}   \right),
  \end{equation}
   and the character $\widetilde \chi : \Gamma_0(N) \to \mathbb C^\times$ defined by
 \begin{equation}\label{e:CharacterForGm_0(N)}
 \widetilde\chi\left(\begin{pmatrix} a&b\\c&d\end{pmatrix}\right) := \chi(d).\end{equation}

An automorphic function of weight $k$, and character $\chi$ 
for $\Gamma_0(N)$ is 
a smooth function $f:  \mathfrak h \to \mathbb C$
 which satisfies the automorphy relation
 \begin{equation}\label{e:automorphy} \left(f \big |_k \gamma\right)(z) \; = \; \widetilde\chi(\gamma)  f(z), \qquad (z \in \mathfrak h)
 \end{equation}
 for all $\gamma   \in \Gamma_0(N)$. 
   
Now fix $\nu\in\mathbb C$. A Hecke-Maass form of weight $k$, type $\nu$, level $N$ and character $\chi\pmod{N}$ is an automorphic cuspidal function of weight $k$,  level $N$ and character $\chi$ which has moderate growth and which is also an eigenfunction of the weight $k$ Laplace operator,
	$$\Delta_k = -y^2\left(\frac{\partial^2}{\partial x^2}+\frac{\partial^2}{\partial y^2}\right)+iky\frac{\partial}{\partial x},$$
with  eigenvalue $\nu(1-\nu)$.

Let $f$ be an automorphic function of weight $k$, level $N$ and character $\chi$. For any positive integer $n$, the Hecke operator (twisted by a character $\chi$) is denoted $T^\chi_n$, and is defined by
	$$\left(T^\chi_n f\right)(z):=\frac{1}{\sqrt{n}}\sum_{ad=n, a, d>0}\chi(a) \sum_{b=0}^{d-1} \left(f\bigg|_k\begin{pmatrix} a & b\\ 0 & d\end{pmatrix}\right)(z).$$
Note that every Dirichlet character modulo $N,$ including
the (``trivial'') principal character, vanishes on 
integers which are not relatively prime to $N.$  This will kill some
of the terms in the definition of $T^\chi_n$ when $(n,N)\ne 1.$
Notably, if $q$ is a prime dividing $N,$ then 
$$\left(T^\chi_q f\right)(z):=\frac{1}{\sqrt{q}} \sum_{b=0}^{q-1} \left(f\bigg|_k\begin{pmatrix} 1 & b\\ 0 & q\end{pmatrix}\right)(z),$$
which coincides with the operator denoted $U_q^*$ in \cite{AtkinLehner:1970}.

 Cusps are defined to be elements of $\mathbb Q\cup\{\infty\}.$  The group $SL(2, \mathbb Z)$ permutes the cusps transitively.   Thus, 
given a cusp $\mathfrak a$ it is possible to choose a matrix $\gamma_\mathfrak a \in SL(2, \mathbb Z)$ such that
   $\gamma_\mathfrak a \infty = \mathfrak a.$  
   The matrix $\gamma_\mathfrak a$ is unique up to an element of the stabilizer,  $\Gamma_\infty,$ 
   of $\infty$ on 
   the right.  
   This group is given explicitly by 
     $$\Gamma_\infty = \left\{ \left.
 \begin{pmatrix} \epsilon & n \\ 0 & \epsilon \end{pmatrix} \right|
 \epsilon \in \{ \pm 1\} , \; n \in \mathbb Z
 \right\},$$
   and is contained in the group $\Gamma_0(N)$ for every $N.$ 
   In particular, 
   $\gamma_\mathfrak a \delta \gamma_\mathfrak a^{-1}$ is independent of the choice 
   of $\gamma_\mathfrak a$ for each $\delta \in \Gamma_\infty.$
    Let $\Gamma_\mathfrak a = \{\gamma\in \Gamma_0(N) \mid \gamma \mathfrak a = \mathfrak a\}.$  Then $\gamma_\mathfrak a^{-1} \Gamma_\mathfrak a \gamma_\mathfrak a$ is a subgroup 
   of finite index in $\Gamma_\infty,$ and contains  the scalar matrix $-1.$ As such it is the product 
   of the group of order $2$ generated by $-1$ and an infinite cyclic group generated by 
   $\left(\begin{smallmatrix} 1& m_\mathfrak a \\ 0& 1\esm$ for some positive integer $m_\mathfrak a.$  This integer
   $m_\mathfrak a$ may be characterized as the least positive power of 
   $\gamma_\mathfrak a\left(\begin{smallmatrix} 1&1\\0&1 \esm\gamma_\mathfrak a^{-1}$ which is in $\Gamma_0(N),$ and 
   as such is independent of the choice of $\gamma_\mathfrak a.$ 
   Let 
   \begin{equation}\label{e:sigma_a}
   \sigma_{\mathfrak a} = \gamma_{\mathfrak a} \begin{pmatrix} \sqrt{m_{\mathfrak a}} & 0 \\ 0 & \sqrt{m_{\mathfrak a}}^{-1}\end{pmatrix},
   \end{equation}
   then $\sigma_\mathfrak a$ has the following key properties:
   $$\sigma_\mathfrak a \infty = \mathfrak a,   \qquad \sigma_\mathfrak a^{-1} \Gamma_\mathfrak a \sigma_\mathfrak a = \Gamma_\infty.$$
The matrix $\sigma_\mathfrak a \left( \begin{smallmatrix} 1&1\\0&1 \end{smallmatrix}
      \right)\sigma_\mathfrak a^{-1} \in \Gamma_\mathfrak a$ is independent of the choice of $\sigma_\mathfrak a.$  
      We denote this element  by $g_\mathfrak a.$  It generates $\Gamma_\mathfrak a$ 
      together with the scalar matrix $-1.$	We  define the cusp parameter $0 \le \mu_\mathfrak a <1$ determined by the condition
   \begin{equation}\label{e:cuspParameter}
   \widetilde\chi(g_\mathfrak a) \; = \; e^{2\pi i \mu_\mathfrak a}.\end{equation}

Let $f$ be an automorphic function of weight $k$, level $N$ and character $\chi$. Then $f$ is cuspidal if
	$$\int_0^1 \left(f|_k\sigma_{\mathfrak a}\right)(x+iy)dx=0$$
for any cusp $\mathfrak a$. 
If $f$ is a 
Maass form  of weight $k$, type $\nu$, level $N$, and character $\chi,$ and $m$ is a positive integer, then $g(z) := f(mz)$ is a 
Maass form  of weight $k$, type $\nu$, level $mN$
and character $\chi \cdot \chi_{0,mN},$ where $\chi_{0, mN}$ is the 
principal character modulo $mN.$  
It may be regarded as a form of weight $k$ and type $\nu$ for 
$\Gamma_0(M),$ where $M$ is any multiple of $mN.$
Forms which arise in this
manner are said to be old, or to be oldforms.  A  
Maass cusp form $f$   of weight $k$, type $\nu$, level $N$, and character $\chi$ is said to be new or a newform if it is orthogonal
(with respect to the Petersson inner product)
to every oldform.  The main result of Atkin-Lehner theory \cite{AtkinLehner:1970} is that the space of newforms has a basis where each basis element is an eigenfunction of all the Hecke operators.  Such newforms are 
called Maass-Hecke newforms. 
 
We remark that a Maass form of level $N$ is a Masss-Hecke newform if 
and only if it is an eigenfunction of $T_p^\chi$ for every prime $p,$
and of the operator $f(z) \mapsto f( \overline{-1/N\bar z}).$ 
In the holomorphic case, this follows from Theorem 9 of \cite{Li:1975}.
The same proof works in the non-holomorphic case.  
%
%
%
%

 It was shown in \cite{Maass:1947}, \cite{Maass:1983}  that a newform of weight $k$, level $N$, and character $\chi\pmod{N}$ 
   for $\Gamma_0(N)$ has a Fourier-Whittaker expansion at every cusp. The expansion takes the following explicit form.
   
\begin{proposition}\label{p:FourierWhittakerExp}{\bf (Fourier-Whittaker expansion at a cusp)}
Let $f$ be a Maass form of weight $k$, type $\nu$, level $N$, and character $\chi\pmod{N}$ for $\Gamma_0(N).$ Let $\mathfrak a \in \mathbb Q\cup\{\infty\}$ be a cusp and let $\sigma_\mathfrak a, \mu_\mathfrak a$ be as in \eqref{e:sigma_a}, \eqref{e:cuspParameter}, respectively. Then
    $$\left(f \big |_k \sigma_\mathfrak a \right)(z) = \sum_{n+\mu_\mathfrak{a}\ne 0} A(\mathfrak a,n) \, W_{\frac{\text{\rm sgn}(n) k}{2}, \; \nu-\frac12} \Big( 4\pi |n+\mu_\mathfrak a|\cdot y  \Big)e^{2\pi i(n+\mu_\mathfrak a)x}, \qquad (z=x+iy\in\mathfrak h)$$
  where $A(\mathfrak a, n) \in \mathbb C$ is called the $n^{th}$ Fourier coefficient at the cusp $\mathfrak a$ and
  $$W_{\alpha,\nu}(y) = \frac{y^{\nu+\frac12} e^{-\frac{y}{2}}}{\Gamma\left(\nu-\alpha+\frac12\right)} \;\int_0^\infty e^{-yt} \, t^{\nu-\alpha-\frac12} \left(1+t   \right)^{\nu+\alpha-\frac12} \; dt, \qquad (\alpha \in \mathbb R, \; \nu \in \mathbb C)$$
 is the Whittaker function.

	Assume $f$ is normalized so that $A(\infty, 1) = 1$. 
If $n=\text{\rm sgn}(n)p_1^{m_1}p_2^{m_2}\cdots p_r^{m_r}$ for distinct primes $p_1, \cdots, p_r$ and positive integers $m_1, \cdots, m_r$, then we have
	$$A(\infty, n) = A(\infty,\text{\rm sgn}(n))\prod_{i=1}^r A(\infty, p_i^{m_i}).$$
\end{proposition} 

\vskip 10pt \noindent
{\bf Remark:}  
In general, the coefficients $A(\mathfrak a, n)$ depend on the
choice of $\gamma_\fa$ and not only on the choice of $\fa.$
\vskip 10pt

\section{Whittaker functions for the adelic lift of a newform}
\label{s:WhittakerAndNewform}
Let $\mathbb A$ be the ring of adeles over $\mathbb Q$. A place of $\mathbb A$ is defined to be either a rational prime or $\infty.$  For a finite prime $p$ and for $x\in \mathbb Q_p$ let
	$$\{x\} :=\begin{cases} \sum\limits_{i=-k}^{-1} a_ip^i, & \hbox{ if } x=\sum\limits_{i=-k}^\infty a_i p^i \in \mathbb Q_p, \hbox{ with } k>0, \;0\leq a_i\leq p-1;\\
	0, & \hbox{ otherwise. }\end{cases}$$
We now define  an additive character $e_v$  at each place $v$. If $x\in \mathbb Q_v$, we let
	$$e_v(x):=\begin{cases} e^{2\pi i x}, &\hbox{ if }v=\infty; \\
					e^{-2\pi i\{x\}}, &\hbox{ if }v<\infty.\end{cases}$$
Furthermore, for $x=\{x_v\}_v\in \mathbb A$, define a global additive character for $\mathbb A$ as
	$$e(x) := \prod_{v\leq \infty}e_v(x_v).$$				
		
Let $\chi$ be a  Dirichlet character modulo $N.$ 
As is well known, there is a character 
 $\chi_{_{\text {\rm idelic}}}:\mathbb{Q}^\times\backslash \mathbb{A}^\times \rightarrow \mathbb{C}^\times$ 
 associated to $\chi$ which we shall refer to
 as the idelic lift of $\chi.$  For the convenience of the reader we repeat the definition

\vskip 10pt\noindent
\begin{definition}{\bf(Idelic lift of a Dirichlet character)} 
Let $\chi$ be a Dirichlet character of conductor $p^f$ 
where $p^f$ is a fixed  prime power. We define the idelic lift of
 $\chi$ to be the unitary Hecke character
 $\chi_{_{\text{\rm idelic}}} : \mathbb Q^\times\backslash\mathbb A^\times \to \mathbb C^\times$ defined as
 $$\chi_{_{\text{\rm idelic}}}(g) := \chi_\infty(g_\infty)\cdot \chi_2(g_2)\cdot \chi_3(g_3) \cdots \; , \qquad \big(g = \{g_\infty, g_2, g_3, \ldots\} \in \mathbb A^\times\big),$$
 where
 $$\chi_\infty(g_\infty) = \begin{cases} \phantom{-}1, &\;\;\; \chi(-1) = 1,\\
 \phantom{-}1, &\;\;\; \chi(-1) = -1, \; g_\infty > 0,\\
 -1, &\;\;\; \chi(-1) = -1, \; g_\infty < 0,
 \end{cases}$$
 and where
$$\chi_v(g_v) = \begin{cases} \chi(v)^{m}, & \text{if} \;\; g_v \in v^m\mathbb Z_v^\times  \;\;
 \text{and}\;\; v \ne p,\\
  \chi(j)^{-1},& \text{if} \;\; g_v \in p^k\left(j + p^f\mathbb Z_p\right) \; \text{with} \;j, k \in \mathbb Z, \; (j,p)=1\; \text{and} \; v = p.\end{cases}$$
  More generally, every Dirichlet character $\chi$ of conductor
   $q \; = \; \prod\limits_{i=1}^r p_i^{f_i},$
   where $p_1, p_2, \ldots p_r$ are distinct primes and $f_1, f_2, \ldots f_r \ge 1$, can be factored as
   $\chi \; = \; \prod\limits_{i=1}^r \chi^{(i)},$
   where $\chi^{(i)}$ is a Dirichlet character of conductor $p_i^{f_i}.$ It follows that $\chi$ may be lifted to a Hecke character $\chi_{_{\text{\rm idelic}}}$ on $\mathbb A^\times$ where
   $\chi_{_{\text{\rm idelic}}} = \prod\limits_{i=1}^r \chi_{_{\text{\rm idelic}}}^{(i)}.$ 
 \end{definition}

\vskip 10pt

For each $p|N$, let
	$$I_{p, N} =\left\{ \begin{pmatrix} a& b\\ c& d\end{pmatrix}\in GL(2, \mathbb{Z}_p)\;\bigg|\; c \equiv 0 \pmod{N}\right\}.$$ 
	Here $c \equiv 0 \pmod{N}$ means $c \in N\cdot \mathbb Z_p.$
	Then we can define
	$$K_0(N):=\left(\prod_{p|N}I_{p, N}\right)\left(\prod_{p\nmid N}GL(2, {\mathbb Z}_p)\right).$$
	For each $p| N$, define a character $\widetilde{\chi_p}: I_{p, N}\rightarrow \mathbb{C}^\times$, where 
	$$\widetilde{\chi_p}\left(\begin{pmatrix} a_p & b_p \\ c_p & d_p\end{pmatrix}\right):= \chi_p(d_p).$$ Define a character $\widetilde{\chi}_{_{\rm idelic}}: K_0(N)\rightarrow \mathbb{C}^\times$ such that
	$$\widetilde{\chi}_{_{\rm idelic}}(k) := \prod_{p|N}\widetilde{\chi_p}(k_p)$$
	for all $k=\{k_p\}_p\in K_0(N)$. 
	If $\chi$ is primitive, then the kernel of $\widetilde{\chi}_{_{\text \rm {idelic}}}$ is given by
			$$\text{Ker}\left(\widetilde{\chi}_{_{\text{\rm idelic}}}\right ) = \prod_{p\nmid N} GL(2, \mathbb{Z}_p)\cdot \prod_{p\mid N} \left\{\begin{pmatrix} a & b\\ c& d\end{pmatrix} \in I_{p, N}\;\bigg|\; d\equiv 1 \left(\hskip -5pt\mod{|N|_p^{-1}}\right)\right\}.$$

Let $\mathbb{A}_{_{\rm finite}}$ denote the finite adeles. 
For $g=\{g_v\}_v\in GL(2, \mathbb{A})\text{ or } GL(2, \mathbb{A}_{_{\rm finite}})$, define $(g)_v := g_v$ where $v$ is a place of $\mathbb{Q}$. 
Then for  $v\le\infty,$  we have the inclusion map
	$i_v: GL(2, \mathbb Q_v)\to GL(2, \mathbb A)$
	defined by
	$$(i_v(g_v))_w:=\begin{cases} \left(\begin{smallmatrix} 1 & 0 \\ 0 & 1\end{smallmatrix}\right), & \hbox{ if } w\neq v;\\
				\;\;\,	g_v, & \hbox{ if }w=v.\end{cases}$$
				
				We shall define the diagonal embedding map
  $i_{_{\rm diag}} : GL(2, \mathbb Q)  \to GL(2, \mathbb A_\mathbb Q)$
by  
  $$i_{_{\rm diag}}(\gamma) \; :=  \; \{\gamma, \gamma, \gamma, \;\;\ldots\}, \quad\qquad (\forall \gamma \in GL(2, \mathbb Q)).$$ 
					
By strong approximation, it follows that for any $g\in GL(2, \mathbb{A})$, there exist $\gamma\in GL(2, \mathbb{Q})$, $g_\infty\in GL(2, \mathbb{R})^+$, $k\in K_0(N)$ such that
	$$g=i_{_{\rm diag}}(\gamma)i_\infty(g_\infty)i_{_{\text{\rm finite}}}(k),$$
	where $i_{_{\text{\rm finite}}}$ denotes the diagonal embedding of $GL(2, \mathbb Q)$ into $GL(2, \mathbb A_{_{\text{\rm finite}}}),$ the group of finite adeles.
	
Let $f$ be a Maass form of weight $k$, type $\nu$, level $N$ and  character $\chi$ modulo $N$. By the above decomposition, we may define a function $f_{_{\text{adelic}}}: GL(2, \mathbb{A})\rightarrow\mathbb{C}$ as
	\begin{equation}\label{e:AdelicLift}
	f_{_{\text{\rm adelic}}}(g) := f_{_{\text{\rm adelic}}}(i_{_{\rm diag}}(\gamma)i_{\infty}(g_\infty)i_{_{\text{\rm finite}}}(k)) := \left(f\big |_k g_\infty\right)(i)\cdot\widetilde{\chi}_{_{\text{\rm idelic}}}(k).
	\end{equation}
	It follows from \eqref{e:automorphy} that $f_{_{\text{adelic}}}$ is well-defined.
Further, $f_{_{\text{adelic}}}$  is an adelic automorphic form with a central character $\chi_{_{\text {\rm idelic}}}$.

	The function $f_{_{\text{adelic}}}$ has a Fourier expansion, 
		$$f_{_{\text{adelic}}}(g) = \sum_{\alpha\in\mathbb{Q}^\times}W_f\left(\begin{pmatrix} \alpha & 0 \\ 0 & 1\end{pmatrix} g\right)$$
		where
		$$W_f(g) := \int_{\mathbb{Q}\backslash\mathbb{A}}f_{_{\text{adelic}}}\left(\begin{pmatrix} 1 & u \\ 0 & 1\end{pmatrix} g\right)e(-u)du$$ 
	for all $g\in GL(2, \mathbb{A})$
	(cf. \cite{Jacquet:1970}, proof of Lemma 11.1.3). The function $W_f(g)$ is called a global Whittaker function for $f_{_{\text{adelic}}}$. 
\vskip 10pt\noindent
\begin{theorem}\label{t:GlobalWhittaker}
Let $f:\Gamma_0(N)\backslash \mathfrak h\to \mathbb C$ be a Maass form of weight $k$, type $\nu$, level $N,$ and  character $\chi\pmod{N}$. 
Let $S$ be a set of representatives for the $\Gamma_0(N)$-equivalence classes of cusps.
Then $f$ has the Fourier-Whittaker expansion at every cusp as in  Proposition~\ref{p:FourierWhittakerExp}. 
 By the Iwasawa decomposition, every $g\in GL(2, \mathbb{A})$ has a decomposition 	
 $$g=\begin{pmatrix} 1 & x\\ 0 & 1\end{pmatrix} \begin{pmatrix} y & 0 \\ 0 & 1\end{pmatrix} \begin{pmatrix} r & 0\\ 0 & r\end{pmatrix} i_{_{\rm diag}}\left(\begin{pmatrix} \epsilon & 0 \\ 0 & 1\end{pmatrix}\right) k,$$
where  $x\in \mathbb{A}, \; r, \,  y=\{y_v\}_v\in \mathbb A^\times,$ with $y_\infty >0,$  $k=\{k_v\}_v\in K=SO(2, \mathbb{R})\prod_{p<\infty}GL(2, \mathbb{Z}_p),$ $k_\infty=\left(\begin{smallmatrix} \cos\theta & \sin\theta\\ -\sin\theta & \cos\theta\esm, \; (0\leq \theta< 2\pi)$, and $\epsilon=\pm 1$.
Let $t:=\prod_p |y_p|_p^{-1}$. Then
	$$W_f(g) = \begin{cases} A\big(\mathfrak{a}, \epsilon m_\mathfrak{a}t-\mu_\mathfrak{a}\big) \,W_{\frac{\epsilon k}{2}, \nu-\frac{1}{2}}(4\pi y_\infty) \, \chi_{_{\text{\rm idelic}}}(r) \, e(x) \, e_{\infty}\left(\frac{k\theta}{2\pi}+tj\epsilon\right) \, \widetilde{\chi}_{_{\text{\rm idelic}}}(k_0),\\
	\hskip 140pt \text{ if } \;\epsilon m_\mathfrak{a} t-\mu_\mathfrak{a}\in\mathbb{Z},\\
	0, \;\;\;\;\text{ otherwise.}\end{cases}$$
Here $k_0\in K_0(N)$, the cusp  $\mathfrak{a}\in S$, and an integer $0\leq j<m_\mathfrak{a}$ are uniquely determined by
	$$i_{_{\text{\rm finite}}}\left(\gamma_\mathfrak{a}\begin{pmatrix} 1 & j\\ 0 & 1\end{pmatrix}\right)\prod_{p|N}i_p\left (\begin{pmatrix} t^{-1}y_p & 0 \\ 0 & 1\end{pmatrix} k_p\right)=k_0 \; \in \; K_0(N).$$
\end{theorem}
\begin{proof}
See \cite{GoldfeldHundley}. 
\end{proof}

\vskip 10pt\noindent 
\begin{definition}\label{d:LocalWhittaker}
Let $f:\Gamma_0(N)\backslash \mathfrak h\to \mathbb C$ be a Maass form of weight $k$, type $\nu$, level $N$ and  character $\chi\pmod{N}$. For each place $v$, we define a function $W_{f, v}(g_v):GL(2, \mathbb Q_v) \to \mathbb C$ as follows.
\vskip10pt
$\bullet$ $\underline{v=\infty}$:
	$$W_{f, \infty}(g_\infty) \; := \;W_f\big(i_\infty(g_\infty)\big)$$
\vskip5pt
$\bullet$ $\underline{v=p<\infty}$: 
	$$W_{f, p}(g_p) \; := \lim_{y_\infty \to 1}\; \frac{W_f\Big(i_\infty\left( \begin{pmatrix} y_\infty& 0\\ 0 & 1\end{pmatrix}\right) \cdot i_p(g_p)\Big)}{W_{\frac{k}{2}, \nu-\frac{1}{2}}(4\pi y_\infty)}.$$
\end{definition}
	\begin{remark}
	The purpose of the limit in the definition of $W_p$ is to make sure that
	$W_p$ is defined even when 
	$W_{\frac{k}{2}, \nu-\frac{1}{2}}(4\pi y_\infty)$
	happens to vanish at $y_\infty=1.$  In 
	fact, the ratio $\frac{W_f\left(i_\infty\left( \bspm y_\infty& 0\\ 0 & 1\espm\right) \cdot i_p(g_p)\right)}{W_{\frac{k}{2}, \nu-\frac{1}{2}}(4\pi y_\infty)}$ will turn out to be the same for all $y_\infty >0$ 
	such that $W_{\frac{k}{2}, \nu-\frac{1}{2}}(4\pi y_\infty)\ne 0.$
	\end{remark}
\vskip 5pt
\noindent
\begin{corollary}\label{c:LocalWhittaker}
Let $f$ be a Maass form of weight $k$, type $\nu$, level $N$ and  character $\chi\pmod{N}$. 
Let $S$ be a set of representatives for the $\Gamma_0(N)$-equivalence classes of cusps.
Let $\mathfrak a\in \mathbb Q$ be a cusp for $\Gamma_0(N)$. For each integer $n$, let $A(\mathfrak a, n)$ be the $n$-th Fourier coefficient of $f$ at the cusp $\mathfrak a$ as in Proposition~\ref{p:FourierWhittakerExp}. 
For every place $v$ of $\mathbb{Q}$ and any $g_v\in GL(2, \mathbb Q_v)$, we have a decomposition 
$$g_v = \begin{pmatrix} 1 & x_v\\ 0 & 1\end{pmatrix} \begin{pmatrix} y_v & 0 \\ 0 & 1\end{pmatrix} \begin{pmatrix} r_v & 0 \\ 0 & r_v\end{pmatrix} k_v,$$ 
where $x_v\in \mathbb{Q}_v$ and  $y_v, \; r_v\in \mathbb{Q}_v^\times$. If $v=\infty$, then  $y_\infty >0$, $\epsilon=\pm 1$ and $k_\infty\in \left(\begin{smallmatrix} \epsilon & 0 \\ 0 & 1\esm SO(2, \mathbb{R})$. If $v$ is finite, then $k_v\in GL(2, \mathbb{Z}_v)$.

For each finite prime $p$, fix $y_p\in \mathbb Q_p^\times$ and $k_p\in GL(2, \mathbb Z_p)$. Then there exists a cusp $\mathfrak a_p\in S$, an integer $0\leq j_p<m_{\mathfrak a_p}$, and $k_{p, 0}\in K_0(N)$, which are uniquely determined by $y_p$ and $k_p$ such that
	$$i_{_{\text{\rm finite}}}\left(\gamma_{\mathfrak a_p}\begin{pmatrix} 1 & j_p \\ 0 & 1\end{pmatrix}\right)i_p\left(\begin{pmatrix} |y_p|_p y_p & 0 \\ 0 & 1\end{pmatrix} k_p\right)=k_{p, 0}\in K_0(N).$$

Then by Definition~\ref{d:LocalWhittaker}, for each place $v$ for $\mathbb Q$, we have the following:
\begin{itemize}
\item $\underline{v=\infty}$: 
$$W_{f, \infty}(g_\infty) = A(\infty, \epsilon) \, W_{\frac{\epsilon k}{2}, \; \nu-\frac{1}{2}}(4\pi y_\infty) \, \chi_\infty(r_\infty) \, e_\infty\left(x_\infty + \frac{k\theta}{2\pi}\right).$$
\item $\underline{v=p\nmid N}$: 
	$$\begin{aligned} W_{f, p}(g_p) & =\begin{cases}  A\big(\infty, |y_p|_p^{-1}\big) \, \chi_p(r_p) \, e_p(x_p), &\text{ if } y_p\in\mathbb{Z}_p,\\
		0,  &\text{ otherwise.}\end{cases}\end{aligned}$$
\item $\underline{v=p\mid N}$:
	$$\begin{aligned} W_{f, p}(g_p) & =\begin{cases} A\Big(\mathfrak{a}_p, m_{\mathfrak{a}_p}|y_p|_p^{-1}-\mu_{\mathfrak{a}_p}\Big) \,\chi_p(r_p) \, e_p(x_p) \, e_\infty\left(|y_p|_p^{-1}j_p\right)\,\widetilde{\chi}_{_{\text{\rm  idelic}}}(k_{p, 0}),\\
	\hskip 140pt	\text{ 		if } \; m_{\mathfrak{a}_p}|y_p|_p^{-1}-\mu_{\mathfrak{a}_p}\in\mathbb{Z},\\
		0, \text{		otherwise.}\end{cases}\end{aligned}$$
\end{itemize}
\end{corollary}
\begin{proof}
Use Definition~\ref{d:LocalWhittaker} and  Theorem~\ref{t:GlobalWhittaker}.
\end{proof}

	\vskip10pt \noindent
{\bf Remark }{\it In Corollary~\ref{c:LocalWhittaker}, assume that $f$ is a normalized Maass form, i.e., $A(\infty, 1) = 1$. Then for $p\nmid N$, 
	$$W_{f, p}(k_p) = 1, \text{ for all }k_p\in GL(2, \mathbb{Z}_p),$$
	and for $p\mid N$, 
	$$W_{f, p}(k_p) = 1, \text{ for all } k_p\in I_{p, N}.$$}
	
	The following theorem is well known, but it is usually not presented in the explicit form which we require for our purposes.

\vskip 10pt 	
\noindent
\begin{theorem} \label{t:FactorizationOfWhittaker}Let $f:\Gamma_0(N)\backslash \mathfrak h\to \mathbb C$ be a Maass-Hecke newform of weight $k$, type $\nu$, level $N$ and  character $\chi\pmod{N}$. If $f$ is  normalized, i.e., its first Fourier coefficient at $\infty$ is $1$, then 
	$$W_f(g) = \prod_v W_{f, v}(g_v).$$ 
\end{theorem}

\begin{proof}
If $f$ is a Maass-Hecke newform then $f_{_{\text{adelic}}}$ generates an irreducible subspace
	$$V \; \subset \; \mathcal{A}_0\Big(GL(2, \mathbb{Q})\backslash GL(2, \mathbb{A}), \;\, \chi_{_{\text{\rm idelic}}}\Big),$$
	and an irreducible automorphic cuspidal representation $(\pi, V)$, 
under the actions of $GL(2, \mathbb{A}_{_{\rm finite}})$ and $(\mathfrak{g}, K_\infty)$ for $GL(2, \mathbb{R})$, where $\mathfrak{g} = \mathfrak{gl}(2, \mathbb{C})$ and $K_\infty = O(2, \mathbb{R})$
(see \cite{Bump:1997} Theorem 3.6.1). By \cite{Flath:1979} there are local representations $(\pi_v, V_v)$ for $GL(2, \mathbb{Q}_v)$ for each place $v$ of $\mathbb{Q}$ such that 
\vskip 5pt
$\bullet$ $\pi_\infty$ is an irreducible and admissible $(\mathfrak{g}, K_\infty)$-module;
\vskip 5pt
$\bullet$ $\pi_v$ is an irreducible and admissible representation for $GL(2, \mathbb{Q}_v)$ for all finite places $v$. Furthermore, for almost all $v$, we know that $V_v$ contains a non-zero $K_v$-fixed vector where $K_v  = GL(2, \mathbb{Z}_v),$ and
 $$(\pi, V)\; \cong \;\bigotimes_v' \, (\pi_v, V_v), \qquad(\text{restricted tensor product}),$$ 
where the restricted tensor product is taken with respect to some choice of 
non-zero $K_v$-fixed vectors in all but finitely many of the spaces $V_v.$ 
(Different choices may give rise to different restricted tensor products, but the representations
obtained are all isomorphic to one another and to $(\pi, V).$)

We shall use the existence and uniqueness of Whittaker models for $(\pi, V)$ and for each of the 
local representations $(\pi_v, V_v)$  (see \cite{Bump:1997} section 3.5).  
For each place $v$, let  $\mathcal{W}(\pi_v, e_v)$ denote the Whittaker space, corresponding to $(\pi_v, V_v)$ and the additive character $e_v$  introduced at the beginning of this section.   
For each place $v$ such that  $(\pi_v, V_v)$ contains a $K_v$-fixed vector, 
the space of $K_v$-fixed vectors is one dimensional (see \cite{Bump:1997} Theorems 2.4.2 , 4.6.2),
and contains a unique element which takes the value $1$ on all of $K_v$ (the 
existence of such an element in the 
non-archimedean 
case is proved in \cite{Bump:1997}, Proposition 3.5.2;  existence in the Archimedean case
can be proved along the same lines but we do not need it here).

Let $\mathcal W_{\text{tensor}}$ denote the restricted tensor product of the spaces $\mathcal W( \pi_v, e_v)$
with respect to the $K_v$-fixed vectors which take the value $1$ on $K_v.$  
Then pointwise multiplication gives an isomorphism between this space and the unique 
Whittaker model of the original representation $\pi.$ 

Indeed, 
suppose that $g = \{g_v\}_v\in GL(2, \mathbb{A})$ and that 
$\otimes _v W_v$ is an element of $\mathcal W_{\text{tensor}},$ so that 
for each $v,$ the Whittaker function  $W_v$ is an element of $\mathcal W( \pi_v, e_v),$ and  $W_v(k_v) =1$ for all but finitely 
many $v$ (for all $k_v \in K_v).$ 
Then the infinite product  $\prod_vW_v(g_v)$ is convergent, because all but finitely many of 
its terms are $1.$ 

Let $\prod_v \mathcal{W}(\pi_v, e_v)$ be the space of complex valued functions on $GL(2, \mathbb{A})$ spanned by $\prod_v W_v(g_v)$ where $W_v(g_v)\in \mathcal{W}(\pi_v, e_v)$ such that $W_v(k_v)=1$ for all $k_v\in GL(2, \mathbb{Z}_v)$ for almost all $v< \infty$.  (i.e., where $\otimes_v W_v \in \mathcal W_{\text{tensor}}.$) 

Then $$\bigotimes_v' \; (\pi_v, V_v)\; \cong \; \prod_v \mathcal{W}(\pi_v, e_v)$$ and $\prod_v\mathcal{W}(\pi_v, e_v)$ is a global Whittaker function space. 
\vskip 10pt
Let
$$\mathcal{W}(\pi, e) := \left\{W_\phi(g):=\int_{\mathbb{Q}\backslash \mathbb{A}}\phi\left(\begin{pmatrix} 1 & u \\ 0 & 1\end{pmatrix} g\right)e(-u)du, \text{ for all } g\in GL(2, \mathbb{A})\;\bigg|\;\phi\in V\right\}.$$
Then $\mathcal{W}(\pi, e)$ is a Whittaker model isomorphic to $(\pi, V)$. By the global uniqueness of Whittaker models, 
$$(\pi, V)\cong \mathcal{W}(\pi, e) =\prod_v\mathcal{W}(\pi_v, e_v) \; \cong \;\bigotimes_v' \,(\pi_v, V_v).$$
It follows from uniqueness of the ``local new vector'' at each place \cite{Casselman:1973} that 
 the element of the restricted tensor product $\bigotimes_v V_v$
corresponding to 
  the element $f_{_{\text{adelic}}}$ of $V$ is a pure tensor $\otimes_{v\le\infty}\; \xi_v.$
  It follows that 
  $W_f (g) =    \prod_v W_v(g_v),$ for some  $W_v(g_v)\in \mathcal{W}(\pi_v, e_v),\; (v\le \infty)$ such that $W_v(k_v)=1$ for all $k_v\in GL(2, \mathbb{Z}_v)$ for almost all $v< \infty$.
  The functions $W_v \; (v\le \infty)$ 
  are unique up to scalar.    
  We must show that for each $v \le \infty,$ 
  we can take $W_v = W_{f,v}.$
  
  Because $A(\infty, 1) \ne 0,$ it 
  follows that $W_f( i_\infty(g_\infty)) \ne 0$
  for some $g_\infty \in GL(2, \R),$ and thence that $W_p(I_2) \ne 0$ for each $p < \infty.$  (Here, $I_2$ is the $2\times 2$ identity matrix.)  We may then normalize $W_p$ so that
  its value at $I_2$ is $1,$ and 
  the equality $W_\infty = W_{f, \infty}$ is immediate.  Further, 
    $W_p(g_p) = W_f( i_\infty(g_\infty)i_p(g_p))/W_f( i_\infty(g_\infty)),$ for all $g_p \in GL(2, \Q_p),$ and any $g_\infty \in GL(2, \R)$ such that $W_f( i_\infty(g_\infty))\ne 0.$  From this it is immediate that $W_\infty = W_{f, \infty}.$ 
  \end{proof}
  Theorem \ref{t:FactorizationOfWhittaker}
  is not ``sharp'' in the sense that
 its conclusion is satisfied by the adelic lifts of many forms which are not new, and 
 by many adelic automorphic forms which 
 are not adelic lifts.  
  Combining a careful analysis of the proof 
  of theorem \ref{t:FactorizationOfWhittaker}
  with the strong multiplicity one theorem 
  (lemma 3.1 of \cite{Langlands:1980})
  permits us to sharpen the result.
  \begin{remarks}
  (1) Clearly the adelic lift $f_{_{\text{adelic}}}$ can be defined for a classical Maass form 
  which is not new.  It follows from the strong multiplicity one theorem 
   that the adelic lift generates an irreducible 
  representation of $GL(2, \A)$ if an only 
  if $f$ is a linear combination of oldforms which are all obtained from the same newform.
  
  (2)  
  If $\phi$ is any adelic cusp form $GL(2, \A) \to \C$, then  one may consider the 
 corresponding Whittaker function
  $$W_\phi(g) := \int_{\mathbb{Q}\backslash\mathbb{A}}f_{_{\text{adelic}}}\left(\begin{pmatrix} 1 & u \\ 0 & 1\end{pmatrix} g\right)e(-u)du,$$ 
  and (again by strong multiplicity one) the following conditions are equivalent:
 \begin{itemize}
 \item There exist local Whittaker functions 
 $W_v \in \mathcal{W}(\pi_v, e_v)$ for $v\le \infty$ such that $W_\phi = \prod_v W_v.$
 \item The automorphic form $\phi$ 
 generates an irreducible cuspidal automorphic representation $(\pi, V)$ of $GL(2, \A),$ and 
 corresponds to a pure tensor under the 
 isomorphism $(\pi, V) \to \otimes_v' (\pi_v,V_v).$ 
 \end{itemize}
 
 (3) The formulae 
 $$W_\infty(g_\infty) = W( i_\infty(g_\infty)),
 \qquad  W_p( g_p) = \lim_{y_\infty \to 1} 
 \frac{W(i_\infty\left( \bspm y_\infty &0 \\0&1 \espm i_p (g_p) \right)}{
  W(i_\infty\left( \bspm y_\infty &0 \\0&1 \espm\right)},
 $$
 are valid whenever $W_p(I_2) \ne 0$ for all $p<\infty,$ or, equivalently, whenever
 there exists $g_\infty \in GL(2, \R)$ such that 
 $W( i_\infty( g_\infty)) \ne 0.$
  \end{remarks}
  
\section{A classical consequence of the tensor product theorem}\label{s:ProofOfMain}

The following theorem may be viewed as a translation of the tensor product theorem into a statement about Fourier coefficients of Maass forms. 
It provides  a partial answer to Stark's question.
\begin{theorem}\label{t:3.8} Fix a positive integer $N = q_1^{e_1}\cdots q_{h}^{e_h}$ where  $q_1^{e_1}, \cdots, q_h^{e_h}$ are powers of distinct primes.  Let $S$ be a set of  inequivalent cusps  for $\Gamma_0(N)$. Fix an integer $k \ge 1$ and let $f$ be a normalized Maass-Hecke newform of weight $k$, type $\nu$, level $N$ and  character $\chi = \prod_{1\le i \le h} \chi_i$, where $\chi_i$ is a Dirichlet character$\pmod{q_i^{e_i}}$ for $1\le i\le h.$
 Fix one cusp  $\mathfrak a\in S$ and let $M$ be a positive integer such that   
	$$\begin{aligned} & \epsilon M+\mu_{\mathfrak a} = \epsilon m_{\mathfrak a}p_1^{m_1}\cdots p_n^{m_n}\cdot q_1^{m'_1}\cdots  q_h^{m'_h}, \\
	&  \hskip 100pt \big(\text{for } m_i, m_j' \in \mathbb Z, \; m_i > 0,\; \text{with} \; 1\le i \le n,\; 1 \le j \le h\big),\end{aligned}$$
	where $p_1, \ldots, p_n$ are distinct primes which do not divide $N$. 
	
	Then for each $i=1, \ldots, h$,  there exist a unique cusp $\mathfrak b_i\in S$ and a unique integer $1\leq j_i<m_{\mathfrak b_i}$ such that
	$$\gamma_{\mathfrak b_i}\begin{pmatrix} 1 & j_i \\ 0 & 1\end{pmatrix}=:\begin{pmatrix} a_i & b_i\\ c_i & d_i\end{pmatrix},$$
	$$\gamma_{\mathfrak b_i}\begin{pmatrix} 1 & j_i \\ 0 & 1\end{pmatrix} \begin{pmatrix} \epsilon M_i & 0 \\ 0 & 1\end{pmatrix} \gamma_{\mathfrak a}^{-1} = \begin{pmatrix} a_{q_i} & b_{q_i}\\ c_{q_i} & d_{q_i}\end{pmatrix}, \qquad \left(\text{for} \;M_i :=\frac{\epsilon M+\mu_{\mathfrak a}}{\epsilon m_{\mathfrak a}q_i^{m'_i}}\in\mathbb Q\right),$$
	where $\gamma_{\mathfrak b_i}\in SL(2, \mathbb Z)$ with $\gamma_{\mathfrak b_i}\infty = \mathfrak b_i$. For all $1\le u \le h$ with $u\neq i$, we have $\left(\begin{smallmatrix} a_i & b_i\\ c_i & d_i\esm\in \Gamma_0(q_u^{e_u})$. Let $\delta_i:=\prod_{u\neq i} q_i^{\max(0, -m_u')}$ and $\delta_i'$ be an integer such that $\delta_i'\delta_i\equiv 1\pmod{q_i^{e_i}}$. Then we have $\delta_i a_{q_i}, \; \delta_i b_{q_i}, \; \delta_i c_{q_i}, \; \delta_id_{q_i}\in\mathbb Z$ and $\delta_i c_{q_i}\equiv 0\pmod{q_i^{e_i}}$.
	
	Let $A(\mathfrak a, \epsilon M)$ denote the ${\epsilon M}^{th}$ Fourier coefficient of $f$ as in Proposition~\ref{p:FourierWhittakerExp}. Then we have
	$$\begin{aligned} &A(\mathfrak a, \epsilon M) = A(\infty, \epsilon) \prod_{i=1}^n A\left(\infty, p_i^{m_i}\right)\\
	&\hskip 30pt\cdot \prod_{i=1}^h \left(A\Big(\mathfrak b_i, m_{\mathfrak b_i}q_i^{m'_i}-\mu_{\mathfrak b_i}\Big) \, e_\infty(q_i^{m'_i}j_i)\left(\prod_{u\neq i, \, u=1}^h \chi_u(d_i)^{-1}\right)\chi_i\Big(\delta_i\delta_i'd_{q_i}\Big)^{-1}\right)\end{aligned}$$
if $m_{\mathfrak b_i}q_i^{m'_i}-\mu_{\mathfrak b_i}\in \mathbb Z$ for all $i=1, \ldots, h$. 
Otherwise $A(\mathfrak a, \epsilon M)=0$.  
\end{theorem}
\begin{proof}
Fix one cusp  $\mathfrak{a}\in S$. Then there exists  $\gamma_{\mathfrak a}\in SL(2, \mathbb Z)$ such that $\gamma_{\mathfrak a} \infty = \mathfrak a.$
For $y=\{y_v\}_v\in\mathbb A^\times$, $y_\infty=1$, let $t=\prod_p|y_p|_p^{-1}$ and for each $q\mid N$, let
	$$k_q := \begin{pmatrix} ty_q^{-1} & 0 \\ 0 & 1\end{pmatrix}\gamma_{\mathfrak a}^{-1}.$$ 
Then $k_q\in GL(2, \mathbb Z_q)$ for every $q\mid N$. Let $\epsilon =\pm 1$ and take 
	$$g=\begin{pmatrix} y & 0 \\ 0 & 1\end{pmatrix} i_{diag}\left(\begin{pmatrix} \epsilon & 0 \\ 0 & 1\end{pmatrix}\right)\prod_{q\mid N}i_q(k_q)\in GL(2, \mathbb{A}).$$ 
Then
$$i_{_{\text{\rm finite}}}(\gamma_\mathfrak{a})\prod_{q\mid N}i_q\left(\begin{pmatrix} t^{-1}y_q & 0 \\ 0 & 1\end{pmatrix} k_q\right)\in K_0(N)$$
since 
$$\gamma_\mathfrak{a}\begin{pmatrix} t^{-1}y_q & 0 \\ 0 & 1\end{pmatrix} k_q= \begin{pmatrix} 1 & 0 \\ 0 & 1\end{pmatrix}\in I_{q, N}\text{ for each } q\mid N$$
and $\gamma_{\mathfrak a}\in GL(2, \mathbb Z_p)$ for any prime $p$.
\vskip 5pt
By Theorem~\ref{t:FactorizationOfWhittaker}, we know that 
$$\begin{aligned} W_f&\left(\begin{pmatrix} y & 0 \\ 0 & 1\end{pmatrix} i_{diag}\left(\begin{pmatrix} \epsilon & 0 \\ 0 & 1\end{pmatrix}\right)\prod_{q\mid N}i_q(k_q)\right)\\
 & = W_{f, \infty}\left(\begin{pmatrix} \epsilon & 0 \\ 0 & 1\end{pmatrix}\right)\left(\prod_{p\nmid N}W_{f, p}\left(\begin{pmatrix} y_p & 0 \\ 0 & 1\end{pmatrix} \begin{pmatrix} \epsilon & 0 \\ 0 & 1\end{pmatrix}\right)\right)\\
 &\cdot \left(\prod_{q\mid N}W_{f, q}\left(\begin{pmatrix} y_q & 0 \\ 0 & 1\end{pmatrix}\begin{pmatrix} \epsilon & 0 \\ 0 & 1\end{pmatrix} k_q\right)\right).\end{aligned}$$
Then by Corollary~\ref{c:LocalWhittaker} for each fixed prime $q\mid N$, there exists a cusp  $\mathfrak{b}_q\in S$ and an integer $0\leq j_q<m_{\mathfrak{b}_q}$ which are uniquely determined by 
	$$i_{_{\text{\rm finite}}}\left(\gamma_{\mathfrak{b}_q}\begin{pmatrix} 1 & j_q \\ 0 & 1\end{pmatrix}\right)i_q\left(\begin{pmatrix} |y_q|_qy_q & 0 \\ 0 & 1\end{pmatrix} \begin{pmatrix} \epsilon & 0 \\ 0 & 1\end{pmatrix} k_q\right)=:k_{q, 0}\in K_0(N).$$ 
	This is equivalent to 	
	$$(k_{q, 0})_q = \gamma_{\mathfrak{b}_q}\left(\begin{smallmatrix} 1 & j_q\\ 0 & 1\esm \left(\begin{smallmatrix} \epsilon t|y_q|_q & 0 \\ 0 & 1\esm \gamma_\mathfrak{a}^{-1}\in I_{q, N};$$
	for each prime $q'\mid N$ and $q'\neq q$, 
	$$(k_{q, 0})_{q'} = \gamma_{\mathfrak b_q}\begin{pmatrix} 1 & j_q\\ 0 & 1\end{pmatrix} \in I_{q', N};$$ 
	and for each prime $p\nmid N$, 
	$$(k_{q, 0})_p = \gamma_{\mathfrak b_q}\begin{pmatrix} 1 & j_q\\ 0 & 1\end{pmatrix} \in GL(2, \mathbb Z_p).$$

It follows form Corollary~\ref{c:LocalWhittaker}, that 
	$$\begin{aligned} & W_f\left(\begin{pmatrix} y & 0 \\ 0 & 1\end{pmatrix} i_{diag}\left(\begin{pmatrix} \epsilon & 0 \\ 0 & 1\end{pmatrix}\right)\prod_{q\mid N}i_q(k_q)\right) = A(\mathfrak{a}, \epsilon m_\mathfrak{a} t-\mu_\mathfrak{a})W_{\frac{\epsilon k}{2}, \nu-\frac{1}{2}}(4\pi)\\
& \\
	& =\begin{cases} A(\infty, \epsilon)W_{\frac{\epsilon k}{2}, \; \nu-\frac{1}{2}}(4\pi)\left(\prod\limits_{q\mid N}A\Big(\mathfrak{b}_q, m_{\mathfrak{b}_q}|y_q|_q^{-1} -\mu_{\mathfrak{b}_q}\Big) \, 
	e_\infty\left(|y_q|_q^{-1}j_q\right) \, \widetilde{\chi}_{_{\text\rm {idelic}}}(k_{q, 0})\right)\\
	\cdot \prod\limits_{p\nmid N}A\left(\infty, |y_p|_p^{-1}\right), \\
	\hskip 40pt\text{ if }\epsilon m_\mathfrak{a}t-\mu_\mathfrak{a}\hbox{ and } \; m_{\mathfrak{b}_q}|y_q|_q^{-1}-\mu_{\mathfrak{b}_q}\in\mathbb{Z} \text{ for all } q\mid N;\\
	0, \;\;\;\;\text{ otherwise.}\end{cases}\end{aligned}$$

	Therefore, 
	$$\begin{aligned} & A(\mathfrak{a}, \epsilon m_\mathfrak{a} t-\mu_\mathfrak{a})\\
	& \\
	& =\begin{cases} A(\infty, \epsilon)\left(\prod\limits_{q\mid N}A\Big(\mathfrak{b}_q, m_{\mathfrak{b}_q}|y_q|_q^{-1} -\mu_{\mathfrak{b}_q}\Big) \, e_\infty\left(|y_q|_q^{-1}j_q\right)\,\widetilde{\chi}_{_{\text{\rm idelic}}}(k_{q, 0})\right)\\
	\cdot \prod\limits_{p\nmid N}A\left(\infty, |y_p|_p^{-1}\right), \\
	\hskip 30pt\text{ if }\epsilon m_\mathfrak{a}t-\mu_\mathfrak{a}, \; m_{\mathfrak{b}_q}|y_q|_q^{-1}-\mu_{\mathfrak{b}_q}\in\mathbb{Z} \text{ for all } q\mid N,\\
	0, \;\;\;\;\text{ otherwise.}\end{cases}\end{aligned}$$

 Now let $\epsilon m_\mathfrak a t-\mu_\mathfrak a = \epsilon M\in\mathbb Z$. Then $\epsilon m_\mathfrak a t = \epsilon M+\mu_\mathfrak a$. Since $t=\prod_p |y_p|_p^{-1}$, 
	$$\epsilon M+\mu_{\mathfrak a} =\epsilon m_\mathfrak ap_1^{m_1}\cdots p_n^{m_n} q_1^{m'_1}\cdots q_h^{m'_h}$$
	for distinct primes $p_1, \ldots, p_n$ (different from $q_1, \ldots, q_h$), and $m_1, \ldots, m_n, \; m'_1, \ldots, m'_h\in \mathbb Z$. This means that for each prime $p$, we take
		$$y_p=\begin{cases} p_i^{m_i}, & \hbox{ if } p=p_i \hbox{ for some } i=1, \ldots, n,\\
						q_i^{m'_i}, & \hbox{ if } p=q_i\hbox{ for some } i=1, \ldots, h,\\
						1, & \hbox{ otherwise.}\end{cases}$$
It follows that if $m_i<0$, then $A(\infty, p_i^{m_i})=0$. So assume that   $m_1, \ldots, m_n>0$. Then for each $i=1, \ldots, h$, we use the uniqueness property above for choosing $\mathfrak b_i$ and integers $1\leq j_i< m_{\mathfrak b_i}$.  
\end{proof}

For later purposes, we now describe a convenient set of representatives
for the equivalence classes of cusps in the case when $N=q^e$ is a prime power.
\vskip 10pt \noindent
\begin{lemma}\label{l:4.1}
If  $q$ is a prime, and $e$ a positive integer, then   the set 
$$
  \{0,\infty\} \, \cup \, \left\{  \frac{1}{c_1q^l} \;\bigg|\;
	\lower6pt\hbox{\vbox{\hbox{$1 \le l < e, \ 
	\gcd(c_1,q) = 1,$}{\hbox{$\;\; 1 \le c_1 < \min(q^l,q^{e-l}),$}}}}
	\right\} 
$$
is a set of representatives for the $\Gamma_0(q^e)$-equivalence classes of cusps.
\end{lemma}
\begin{proof}
It is well known and easily verified that the group $SL(2, \mathbb Z)$ permutes the set of 
cusps transitively.  It follows that $\Gamma_0(q^e)$-equivalence classes of cusps 
 are naturally identified with double cosets $\Gamma_0(q^e) \backslash SL(2, \mathbb Z) / \widetilde\Gamma_\mathfrak a,$ where $\tilde \Gamma_\mathfrak a$ denotes the stabilizer  in $SL(2, \mathbb Z)$ of any fixed cusp 
 $\mathfrak a.$  It is convenient to employ this identification with $\mathfrak a = \infty.$ 
 As remarked above, the stabilizer $\widetilde \Gamma_\infty=\Gamma_\infty$ is 
 independent of $N$ and 
 given explicitly by 
  $$\Gamma_\infty = \left\{ \left.
 \begin{pmatrix} \epsilon & n \\ 0 & \epsilon \end{pmatrix} \right|
 \epsilon \in \{ \pm 1\} , \; n \in \mathbb Z
 \right\}.$$
 Now, it follows easily from the definition of $\Gamma_0(q^e)$ that 
 $\Gamma_0(q^e) \backslash SL(2, \mathbb Z)$  is naturally identified with 
 $$B^1( \mathbb Z / q^e \mathbb Z) \backslash SL(2, \mathbb Z/ q^e\mathbb Z),$$ where 
 $B^1(R)$ denotes the group of upper triangular matrices with entries in the ring $R$ 
 and determinant equal to $1.$  
 The projective line $\mathbb P^1 ( \mathbb Z/ q^e \mathbb Z)$ 
 is given by 
 $$\left\{ (x_0, x_1) \in (\Z/q^e\Z)^2 \;|\; \langle x_0, x_1 \rangle  = \Z/q^e\Z\right\}/\sim.$$
 Here, $ \langle x_0, x_1 \rangle$ denotes the ideal generated by $x_0$ and
 $x_1,$ and $\sim$ denotes the equivalence relation given by 
 $$ (x_0, x_1) \sim (x_0',x_1') \iff 
 (x_0',x_1') =  (\lambda x_0, \lambda x_1) , \;\text{some} \;
 \lambda \in (\Z/ q^e \Z)^\times.$$
 We write $[x_0:x_1]$ for the equivalence class of $(x_0, x_1).$
 The group   $SL(2, \mathbb Z/ q^e\mathbb Z),$ has 
 a natural right action on $\mathbb P^1(\mathbb Z/ q^e \mathbb Z)$ given by 
 $$[x_0:x_1] \cdot \begin{pmatrix} a& b \\ c& d \end{pmatrix} 
 = [ax_0 + cx_1: bx_0+dx_1], 
 \qquad [x_0: x_1] \in \mathbb P^1( \mathbb Z/ q^e \mathbb Z) ,   \begin{pmatrix} a& b \\ c& d \end{pmatrix} 
 \in SL(2, \mathbb Z/ q^e \mathbb Z).$$
 Clearly, the stabilizer of $[0:1]$ is $B^1( \mathbb Z/ q^e \mathbb Z).$  Thus 
 $\P^1( \Z/q^e \Z)$ may be identified with the coset space
 $\Gamma_0(q^e) \backslash SL(2, \mathbb Z).$
 It follows that 
 $\Gamma_0(q^e) \backslash SL(2, \mathbb Z) / \Gamma_\infty$ is in one-to-one 
 correspondence with orbits for the action of $\Gamma_\infty$ on  $\mathbb P^1 ( \mathbb Z/ q^e \mathbb Z)$
 via inclusion into $SL(2, \mathbb Z)$ and then projection to $SL(2, \mathbb Z/q^e \mathbb Z).$ 
 Note that the coset in $\Gamma_0(q^e)\backslash SL(2, \mathbb Z)$ which corresponds to the 
 element $[x_0:x_1] \in \mathbb P^1(\mathbb Z/ q^e \mathbb Z)$ consists of all matrices
 $\left(\begin{smallmatrix} a& b \\ c& d \end{smallmatrix}\right)$  such that 
 $(c, d) \equiv (\lambda x_0, \lambda x_1) \pmod{q^e}$ for some $\lambda \in (\mathbb Z/ q^e \mathbb Z)^\times.$

 It is clear that  
 $$\mathbb P^1( \mathbb Z/ q^e \mathbb Z) \;=\;
 \Big\{ [1:x_1] \;\Big |\; x_1 \in \mathbb Z/ q^e \mathbb Z\Big\} 
 \; \cup \; \Big \{ [x_0: 1] \;\Big |\; x_0 \in \mathbb Z/  q^e \mathbb Z- (\mathbb Z/ q^e \mathbb Z)^\times \Big\},$$
 and that the action of $\Gamma_\infty$ permutes the elements of 
 $\{ [1:x_1] \;|\; x_1 \in \mathbb Z/ q^e \mathbb Z\}$ transitively.  
 It follows that the $\Gamma_0(q^e)$-cosets corresponding to these elements comprise a 
 single double coset in $\Gamma_0(q^e) \backslash SL(2, \mathbb Z) / \Gamma_\infty$
 which is represented by 
$\left(\begin{smallmatrix} 0&-1\\ 1&0 \esm.$
This matrix maps $\infty$ the the cusp $0.$ 
 
 We study the action 
 of $\Gamma_\infty$ on $\big\{ [x_0: 1] \;\big |\; x_0 \in \mathbb Z/  q^e \mathbb Z- (\mathbb Z/ q^e \mathbb Z)^\times \big\}.$
 Writing $x_0 = q^l c_1$ with $1\le l \le e, 1\leq c_1< q^{e-l} $ and $\gcd (c_1, q)=1,$
we compute 
$$ \left [ q^lc_1 : 1\right]
\begin{pmatrix}
	1 & n \\ 0 & 1 \end{pmatrix}
  = \left[q^lc_1 : q^lc_1n+1 \right] = 
\left[q^lc_1 \overline{(q^lc_1n+1)} : 1\right ],$$
where $\overline{a}$ denotes $a^{-1}$ modulo $q^{e-l}.$
From this we see at once that each part of the partition 
$$\bigcup_{l = 1}^e\Big\{\left [q^l c_1: 1\right] \;\Big |\; (c_1, q)=1, \; 1\leq c_1< q^{e-l} \Big\} $$
 is preserved by the action of $\Gamma_\infty,$ and that  $\Gamma_\infty$  acts trivially 
 on $$\Big\{ [q^l c_1: 1] \;\Big |\; (c_1, q)=1, \; 1\leq c_1< q^{e-l} \Big\} $$
 whenever $e-l\le  l,$ for in this case $\overline{(q^lc_1n+1)} \equiv 1 \pmod{q^{e-l}},$
whence  $[q^lc_1 \overline{(q^lc_1n+1)} : 1 ]= [q^lc_1:1].$  
When $l = e,  \left\{ [q^l c_1: 1] \;|\; (c_1, q)=1, \; 1\leq c_1< q^{e-l} \right\} = [0:1],$ which corresponds to 
the element of $\Gamma_0(q^e)\backslash SL(2, \mathbb Z)$ represented by the identity matrix.
The corresponding cusp is $\infty.$
For other values of $l \ge \frac e2,$ we have  shown that for each $c_1$ such that 
$1\leq c_1< q^{e-l} $ and $\gcd (c_1, q)=1,$
the coset in 
  $\Gamma_0(q^e)\backslash SL(2, \mathbb Z)$ corresponding to $[ q^lc_1 : 1],$ is in fact a double 
  coset  $\Gamma_0(q^e)\backslash SL(2, \mathbb Z)/ \Gamma_\infty.$
This coset is represented by the matrix 
$\left(\begin{smallmatrix} 1&0 \\ q^l c_1 & 1 \esm $ which maps $\infty$ to $\frac 1{q^lc_1}.$

When $e-l > l,$ the action of $\left(\begin{smallmatrix}
1 & n \\ 0 & 1 \esm$  on $\left\{ [q^l c_1: 1] \;|\; (c_1, q)=1, \; 1\leq c_1< q^{e-l} \right\} $
factors through the function 
$n \mapsto \overline{(q^lc_1n+1)}.$  This maps $\mathbb Z$ into  the group $U_l$ of units in $\mathbb Z/ q^e \mathbb Z$ which are equivalent to $1$ modulo 
$q^l.$   It is easy to see that this function is surjective.  
More precisely $n \mapsto c_1n$ is a bijection $\mathbb Z/ q^{e-l}\mathbb Z\to \mathbb Z/ q^{e-l}\mathbb Z,$
while 
$m \mapsto 1+q^l m$ is a bijection $ \mathbb Z/ q^{e-l}\mathbb Z \to U_l,$ 
and $\overline{\phantom{X}}$ is a bijection $U_l \to U_l.$

Thus we are reduced to studying the action of $U_l$ on $(\mathbb Z/ q^{e-l}\mathbb Z)^\times.$ 
Clearly $c_1u \equiv c_1 \pmod{q^l}$ for all $c_1 \in (\mathbb Z/ q^{e-l}\mathbb Z)^\times,$ and 
$u \in U_l.$  Equally clearly, if $c_1, c_1' \in (\mathbb Z/ q^{e-l}\mathbb Z)^\times,$ and 
$c_1 \equiv c_1' \pmod{q^l}$ then $c_1' \overline{c_1} \in U_l.$  It follows that the orbits for the 
action of $U_l$ on $(\mathbb Z/ q^{e-l}\mathbb Z)^\times,$ are precisely the residue classes modulo
$q^l.$   This completes the proof.
\end{proof}

For a prime $q$ and positive integer $e$, fix $N=q^e$. From Lemma~\ref{l:4.1}, we can take the complete set of inequivalent cusps for $\Gamma_0(q^e)$ as
	\begin{equation}\label{e:CuspReps}
	S \; = \; \{0,\infty\} \, \cup \,\left\{  \frac{1}{c_1q^l} \;\bigg|\;
	\lower6pt\hbox{\vbox{\hbox{$1 \le l < e, \ 
	\gcd(c_1,q) = 1,$}{\hbox{$\;\; 1 \le c_1 < \min(q^l,q^{e-l}),$}}}}
	\right\} .
	\end{equation}
For each cusp $\mathfrak{a}\in S$, we have the following.
$$\begin{cases}
	\mathfrak{a} = 0:& \gamma_\mathfrak{a} = \left(\begin{smallmatrix} 0 & -1 \\ 1 & 0\esm, \sigma_\mathfrak{a} = \left(\begin{smallmatrix} 0 & -1 \\ 1 & 0 \esm\left(\begin{smallmatrix} q^{\frac{e}{2}} & 0 \\ 0 & q^{-\frac{e}{2}}\esm, \; g_{\mathfrak{a}}=\left(\begin{smallmatrix}1 & 0 \\ q^e & 1\esm, \; m_\mathfrak{a} = q^e, \\
	& \\
	\mathfrak{a} =\frac{1}{cq^l}, \; l\leq \frac{e}{2}:& \gamma_\mathfrak{a} = \left(\begin{smallmatrix} 1 & 0 \\ q^l c & 1\esm, \; \sigma_\mathfrak{a} = \left(\begin{smallmatrix} 1 & 0 \\ q^lc & 1\esm\left(\begin{smallmatrix} q^{\frac{e}{2}-l} & 0 \\ 0 & q^{l-\frac{e}{2}}\esm,\\	
	& g_\mathfrak{a} = \left(\begin{smallmatrix} 1-q^{e-l}c & q^{e-2l} \\ -q^ec^2 & 1+q^{e-l}c\esm, \; m_\mathfrak{a} = q^{e-2l},\\
	& \\	
\mathfrak{a} = \frac{1}{cq^l}, \; l>\frac{e}{2}: & \gamma_\mathfrak{a} = \left(\begin{smallmatrix} 1 & 0 \\ q^l c & 1\esm , \; \sigma_\mathfrak{a} = \left(\begin{smallmatrix} 1 & 0 \\ q^lc & 1\esm, \\
	& g_\mathfrak{a} = \left(\begin{smallmatrix} 1-q^lc & 1,\\ -q^{2l}c^2 & 1+q^l c\esm, \; m_\mathfrak{a} = 1.
\end{cases}$$

\noindent
From the above table, we can easily see that $\mu_0=\mu_\infty=0$ since $\widetilde{\chi}(g_0) = \widetilde{\chi}(g_\infty)=1$. If $\mathfrak{a}\neq 0, \infty$, then $\mathfrak{a} = \frac{1}{cq^l}$, with $1\leq c < \min(q^l, q^{e-l})$ and $(c, q)=1$ and 
	$$\widetilde{\chi}(g_\mathfrak{a}) = \chi\big(1+c\cdot\max(q^l, q^{e-l})\big) = \chi\big(1+\max(q^l, q^{e-l})\big)^c = e^{2\pi i \mu_\mathfrak{a}}.$$
	So $\widetilde{\chi}(g_\mathfrak{a})^{\min(q^l, q^{e-l})} = 1$.

	\vskip10pt
\noindent
\begin{lemma}\label{l:5.2}
Fix an integer $e \ge 1.$ Let $\chi$ be a  Dirichlet character of prime power level $N=q^e$. Let $\chi_{_{\text{\rm trivial}}}$ be the trivial character modulo $q^e$.  Let $\chi_0$ be a primitive Dirichlet character of prime power level $N_0=q^{e_0}$  (with $0\leq e_0\leq e$) such that $\chi=\chi_0\cdot \chi_{_{\text{\rm trivial}}}$.  Then the following hold.
\begin{itemize}
\item For each integer $1\leq l<e$ with $\max(q^l, q^{e-l})\geq q^{e_0}$,  and any  cusp $\mathfrak a=\frac{1}{cq^l}$ with $1\leq c <\min(q^l, q^{e-l})$ and $(c, q)=1,$  the cusp parameter $\mu_\fa$ is zero.    
\item For each integer $1\leq l<e$ with $\max(q^l, q^{e-l})<q^{e_0}$, there exists a cusp $\mathfrak a_0=\frac{1}{c_0q^l}$ with cusp parameter $\mu_{\mathfrak a_0}=\min\left(q^{e_0-l}, \, q^{e_0-e+l}\right)^{-1}$.  Here $1\leq c_0 <\min (q^{e_0-l},\, q^{e_0-e+l})$ and $(c_0, q)=1$. 
Then
	$$\widetilde\chi(g_{\mathfrak a_0}) = e^{2\pi i \min\left(q^{e_0-l}, \, q^{e_0-e+l}\right)^{-1}},$$
and for any cusp $\mathfrak a=\frac{1}{cq^l}$ with $1\leq c <\min(q^l,  q^{e-l})$ and $(c, q)=1$, there exists a unique integer $1\leq r< \min\left(q^{e_0-l}, \, q^{e_0-e+l}\right)$ with $(r, q)=1$ such that
	$$c \; \equiv \; rc_0 \;\Big(\hskip -9pt\mod{\min\left(q^{e_0-l}, \, q^{e_0-e+l}\right)}\Big)\;\text{ and } \;\mu_{\mathfrak{a}}=r\mu_{\mathfrak{a}_0}.$$
	\end{itemize}
\end{lemma}
\begin{proof}
For any integer $m$ with $1\leq m<e,$ let $U_m = \{ a \in (\Z/q^e\Z)^\times\mid a \equiv 1 \pmod q^m\}.$
This is a subgroup.  In fact, it is the kernel of the natural projection from $ (\Z/q^e\Z)^\times$
to $ (\Z/q^m\Z)^\times.$  The integer $e_0$ is the smallest integer such that $\chi$ factors through this 
projection.  Thus the restriction $\chi|_{U_m}$ of $\chi$ to $U_m$ is trivial iff $m\ge e_0.$

Use the set of representatives for cusps $S$ in \eqref{e:CuspReps}. 
As we see from the table above, the lower right entry $d_\fa$ of the generator $g_\fa$ 
 is an element of 
$U_{\max{l, e-l}}.$ 
Since $\mu_\fa$ is defined so that $e^{2\pi i\mu_\fa} = \chi( d_\fa),$ we 
need to study  the restriction $\chi|_{U_{\max(l, e-l)}}.$

If $e_0 \le \max(l, e-l),$ this restriction is trivial and $\mu_\fa$ is zero, regardless of $c.$
The function  $c \mapsto 1+cq^{\max(l, e-l)}$ is an 
isomorphism $\Z/q^{\min(l,e-l)}\Z \to U_{\max(l,e-l)}.$  
Composing with 
$\chi,$ we obtain a homomorphism $\varphi$ from $\Z/q^{\min(l,e-l)}\Z$ to $\C.$  
For any $m$ with $\max(l,e-l) \le m \le e,$ the preimage of $U_m$ in  $\Z/q^{\min(l,e-l)}\Z$
is the cyclic subgroup generated by $q^{m-\max(l,e-l)},$ and these 
are the only subgroups of $\Z/q^{\min(l,e-l)}\Z.$
Since the kernel of $\chi$ contains $U_{e_0},$ but not $U_{e_0-1},$ it follows that 
the kernel of $\chi|_{U_{\max{l, e-l}}}$ is precisely equal to $U_{e_0},$ 
and that its image is the $(q^{e_0-\max( l, e-l)})^{\text{\rm th}}$ roots of unity.
Furthermore, $\varphi$ factors through the natural projection $\Z/q^{\min(l, e-l)}\Z
\to \Z/q^{e_0-\max(l,e-l)}.$  For $c_0,$ we take the least positive element 
of the residue class which maps to $e^{2\pi i \min\left(q^{e_0-l}, \, q^{e_0-e+l}\right)^{-1}}.$
\end{proof}

	\vskip 10pt
	The following theorem provides a partial  answer to Stark's question in the case of prime power level.
\vskip 5pt\noindent
\begin{theorem}\label{t:main}
Let $q^e$ be a fixed prime power.  Let $f$ be a Hecke-Maass newform of  level $q^e$, character $\chi\pmod{q^e}$, weight $k$, type $\nu$ for $\Gamma_0(q^e)$. Assume $\chi = \chi_0\cdot \chi_{_{\text{\rm trivial}}}$ where $\chi_{_{\text{\rm trivial}}}$ is the trivial character modulo $q$ and $\chi_0$ is a primitive Dirichlet character of prime power level $q^{e_0}$  (with $0\leq e_0\leq e$). 
For $\mathfrak a \in S$ and $n \in \mathbb Z$, let $A(\mathfrak a, n)$ denote the $n^{th}$ Fourier coefficient of $f$ at the cusp $\mathfrak a$ as in Proposition~\ref{p:FourierWhittakerExp}. Assume that $A(\infty, 1)=1$.  
 For any $\mathfrak a \in S$ and an arbitrary non-negative integer $M$  let
  $$\boxed{\epsilon M+\mu_{\mathfrak a} \; = \; \epsilon m_{\mathfrak a}p_1^{m_1}\cdots p_n^{m_n}\cdot q^m}$$
where  $\epsilon = \pm1,$	 $m_1, \ldots, m_n$ are positive integers,  $m \in \mathbb Z$,    $p_1, \ldots, p_n$ are distinct primes different from $q$, and $\mu_\mathfrak a$ is the cusp parameter given in \eqref{e:cuspParameter}. Set $M_0 = p_1^{m_1}\cdots p_n^{m_n}.$ 
   For each cusp $\mathfrak a = \frac{1}{cq^l} \in S$, there exists a unique cusp $\frac{1}{c' q^l} \in S$, 
determined by the conditions
 $$1\le c'\le \text{\rm min}\left(q^{e-l}, q^l\right), \qquad c'\epsilon M_0 \equiv c \left(\hskip -6pt\mod\text{\rm min}\left(q^{e-l}, q^l\right)\right).$$ 
If $l \le e/2$ then there is, in addition, a 
 unique integer $j$   determined by the conditions$$
0 \le j < q^{e-2l},\qquad
cc'j \; \equiv \; (c'\epsilon M_0 -c)\cdot q^{-l}\hskip -4pt\pmod{q^{e-2l}}.$$
Let $\alpha\ge 0$ denote the greatest integer such that $q^\alpha \mid M.$  
 Then $m, \mu_\fa$ and $A( \fa, \epsilon M)$ are given as follows.
 \begin{description}
\item[$\underline{\bullet\; \mathfrak a = \infty}$]  In this case $\mu_\infty = 0$,  $m = \alpha,$ and 
	$$A(\infty, \epsilon M) =  A(\infty, \epsilon)A(\infty, p_1^{m_1})\cdots A(\infty, p_n^{m_n})A(\infty, q^m).$$
\item[$\underline{\bullet\; \mathfrak a=0}$]  In this case $\mu_0 = 0,$ $m = \alpha - e,$ and 

	$$A(0, \epsilon M) =  A(\infty, \epsilon M_0)A(0, q^{e+m})\chi(\epsilon M_0)^{-1}.$$

\item[$\underline{\bullet\; \mathfrak a=\frac{1}{cq^l},\hbox{ and } \mu_\mathfrak a\neq 0}$] 
In this case
$e_0 > \max(l, e-l),$ 
$m=  -e_0+l,$ and  
$$A\left(\frac{1}{cq^l}, \; \epsilon M\right)=\begin{cases} A(\infty, \epsilon M_0)A\left(\frac{1}{c'q^l}, 0\right)e_\infty\left(q^{m}\cdot j\right)\chi\left(jc'q^l+1\right)^{-1},\\
\hskip 30pt\hbox{ if } l\leq \frac{e}{2},\\
	A(\infty, \epsilon M_0)A\left(\frac{1}{c' q^l},\, 0\right), \\
	\hskip 30pt\hbox{ if } l>\frac{e}{2},
	\end{cases}$$
	Furthermore,  the cusp parameter of $\frac{1}{c'q^l}\in S$ is $\min(q^{e_0-l}, q^{e_0-e+l})^{-1}$. 
	(If $e_0=e,$ i.e., if $\chi$ is primitive, then there is a unique cusp $\fa_0=\frac{1}{c_0q^l}\in S$ having this property, so that
	$c'=c_0,$  independently of $c,\epsilon,$ and $M$!)

\item[$\underline{\bullet\; \mathfrak a=\frac{1}{cq^l}, \hbox{ and }\mu_\mathfrak a=0 }$]
In this case 
$m = \alpha - \max(e-2l, 0)$ and
	$$A\left(\frac{1}{cq^l},\; \epsilon M\right) = \begin{cases} A(\infty, \epsilon M_0) A\left(\frac{1}{c'q^l}, \; q^{e-2l+m}\right)e_\infty\left(q^m\cdot j\right)\chi\left(jc'q^l+1\right)^{-1},\\
	\hskip 30pt\hbox{ if } l\leq \frac{e}{2},\\
		A(\infty, \epsilon M_0)A\left(\frac{1}{c'q^l}, \; q^m\right),\\
		\hskip 30pt\hbox{ if }l>\frac{e}{2}.\end{cases}$$
	\end{description}
		\end{theorem}

\begin{proof}
Fix $\mathfrak a\in S.$  Let $M$ be a positive integer and $\epsilon=\pm 1.$ 
Write 	$$\epsilon M+\mu_{\mathfrak a} = \epsilon m_{\mathfrak a}p_1^{m_1}\cdots p_n^{m_n} q^m,$$
	for distinct primes $p_1, \cdots, p_n\neq q$, and integers $m_1, \cdots, m_n, m$ with $m_1, \cdots, m_n>0$. 
	It follows from Lemma~\ref{l:5.2} and the discussion preceding it that $\mu_\fa=0$ except when 
	$\fa = \frac 1{cq^l}$ with $\max(l,e-l) < e_0.$  Also, in all cases, $m_{\mathfrak a}$ is a power of $q.$
	Thus, when $\mu_\fa=0,$ the expression 
	$$M= p_1^{m_1}\cdots p_n^{m_n}\cdot (m_\fa q^m)$$
	is the prime factorization of $M.$  The expressions for $m$ in terms of $\alpha$ in the various
	cases now follow easily from the values of $m_\fa$ tabulated above.

When $\fa = \frac 1{cq^l}$ with $\max(l,e-l) < e_0,$ it follows from Lemma~\ref{l:5.2} that $\mu_\fa$ is a 
rational number of the form $\frac r{\min\left(q^{e_0-l}, \, q^{e_0-e+l}\right)}=\frac r{ q^{e_0-\max(l,e-l)}}$ with $gcd(r,q)=1.$  Consequently $\epsilon M+\mu_{\mathfrak a}$
is a rational number with the same denominator, and a numerator which is congruent to $r \pmod{\min\left(q^{e_0-l}, \, q^{e_0-e+l}\right)}.$  Since in this case $m_\fa = q^{\max(l, e-l)-l},$ we obtain $m= -e_0+l.$ 

	Let
	$$M_0:= \frac{\epsilon M+\mu_{\mathfrak a}}{\epsilon m_{\mathfrak a}q^m}=p_1^{m_1}\cdots p_n^{m_n}.$$
Then by Theorem~\ref{t:3.8}, there exists a unique cusp $\mathfrak b\in S$ and a unique integer $1\leq j<m_{\mathfrak b}$ such that 
	$$\gamma_\mathfrak b \begin{pmatrix} 1 & j \\ 0 & 1\end{pmatrix}  =: \begin{pmatrix} a & b\\ c & d\end{pmatrix}\in SL(2, \mathbb Z),$$
	$$\gamma_{\mathfrak b}\begin{pmatrix} 1 & j \\ 0 & 1\end{pmatrix} \begin{pmatrix} \epsilon M_0 & 0 \\ 0 & 1\end{pmatrix} \gamma_{\mathfrak a}^{-1} =: \begin{pmatrix} a_q & b_q \\ c_q & d_q\end{pmatrix},$$
	where $c_q \equiv 0 \pmod{q^e}$. 
Then
	$$A(\mathfrak a, \epsilon M) = A(\infty, \epsilon)\prod_{i=1}^n A(\infty, p_i^{m_i})\left(A(\mathfrak b, m_{\mathfrak b}q^m-\mu_{\mathfrak b})e_\infty(q^m\cdot j)\chi(d_q)^{-1}\right).$$
	(we shall show that $m_{\mathfrak b}q^m-\mu_{\mathfrak b}$ is always integral). 

\vskip 10pt	
{\bf (1)} If $\mathfrak a=\infty$, then $\mu_\mathfrak a=0$ and $m_\mathfrak a=1$. Since $\gamma_\infty \left(\begin{smallmatrix} \epsilon M_0 & 0 \\ 0 & 1\esm \gamma_\infty^{-1} = \left(\begin{smallmatrix} \epsilon M_0 & 0 \\ 0 & 1\esm \in I_{q, N}$. So  $\mathfrak b=\infty$.  Since $m = \alpha > 0$ and $\mu_\mathfrak b =0,$ it follows at once that $m_{\mathfrak b}q^m-\mu_{\mathfrak b}\in \Z.$  Furthermore,
	$$A(\infty, \epsilon M) = A(\infty, \epsilon)\prod_{i=1}^n A(\infty, p_i^{m_i})\cdot A(\infty, q^m).$$

\vskip 5pt
{\bf (2)} If $\mathfrak a=0$, then $\mu_\mathfrak a=0$ and $m_\mathfrak a=q^e$. Then $\gamma_0\left(\begin{smallmatrix} \epsilon M_0 & 0 \\  0 & 1\esm \gamma_0^{-1} = \left(\begin{smallmatrix} 1 & 0 \\ 0 & \epsilon M_0 \esm\in I_{q, N}$. Therefore,  $\mathfrak b=0$ and $j=0$. 
Once again $\mu_\fb =0.$  Furthermore,  $m_\fb q^m = q^\alpha \in \Z.$
Finally,
	$$A(0, \epsilon M) = A(\infty, \epsilon)\prod_{i=1}^n A(\infty, p_i^{m_i}) \cdot A(0, q^{e+m}) \chi(\epsilon M_0)^{-1}$$ 
	
If $\mathfrak a\neq 0, \infty$ then $\mathfrak a=\frac{1}{cq^l}$ for some  fixed integers $1\leq l<e$
and  $1\leq c<\min(q^l, q^{e-l})$ with $(c, q)=1$.  Also, $m_{\mathfrak a}=\max(q^{e-2l}, 1)$. 
Let us explicitly determine $\fb$ in this case.  
First, assume $l \ge \frac e2.$
Consider the computation
\begin{equation}\label{e:5.3}
\begin{aligned}\gamma_\mathfrak b\begin{pmatrix} \epsilon M_0 & 0 \\ 0 & 1\end{pmatrix} \gamma_{\mathfrak a}^{-1}& =\begin{pmatrix} 1 & 0 \\ c'q^l & 1\end{pmatrix} \begin{pmatrix} \epsilon M_0 \\ 0 & 1\end{pmatrix} \begin{pmatrix} 1 & 0 \\ -cq^l & 1\end{pmatrix}\\
	&=\begin{pmatrix} \epsilon M_0 & 0 \\ q^l(c'\epsilon M_0 -c) & 1\end{pmatrix}.
	\end{aligned}
	\end{equation}
	It is clear that the matrix on the right-hand side is an element of $I_{q,N}$ if and only if
	$c\equiv c'\epsilon M_0 \pmod{q^{e-l}}.$  Thus $\fb =\frac1{c'q^l}$ for this particular 
	value of $c'.$  Referring to the table above, we see  that $m_\fb = m_\fa=1,$ and 
	$d_\fa \equiv d_\fb^{\epsilon M_0} \pmod{q^e},$ whence $\epsilon M_0 \mu_\fb -\mu_\fa \in \Z.$
	
	Now assume $l< \frac e2.$  	Consider the computation
\begin{equation}\label{e:5.4}
\begin{aligned} \gamma_\mathfrak b\begin{pmatrix} 1 & j \\ 0 & 1\end{pmatrix} \begin{pmatrix} \epsilon M_0 & 0 \\ 0  & 1\end{pmatrix} \gamma_\mathfrak a^{-1} & =\begin{pmatrix} 1 & 0 \\ c'q^l & 1 \end{pmatrix} \begin{pmatrix} 1 & j \\ 0 & 1\end{pmatrix} \begin{pmatrix} \epsilon M_0 & 0 \\ 0 & 1\end{pmatrix} \begin{pmatrix} 1 & 0 \\ -cq^l & 1\end{pmatrix}\\
	&=\begin{pmatrix} \epsilon M_0 -jcq^l & j \\ q^l(c'\epsilon M_0 -jc'cq^l-c) & jc'q^l+1\end{pmatrix}.\end{aligned}\end{equation}
	
	It is clear that the matrix on the right-hand side is an element of $I_{q,N}$ if and only if $c'$ and $j$
	are such that 
	$c\equiv c'\epsilon M_0 \pmod{q^{e-l}}$ and $jc'c\equiv  \left(\frac{c'\epsilon M_0-c}{q^l}\right)
	\pmod q^{e-2l}.$
	This shows that $\fb = \frac{1}{c'q^l},$ where $c'$ is the unique solution to $c\equiv c'\epsilon M_0 \pmod{q^{e-l}}$
	in the range $1 \le c' < q^l.$  It follows at once that $m_\fb = m_\fa=q^{e-2l},$ and that 
	$\epsilon M_0 \mu_\fb -\mu_\fa \in \Z.$
	
\vskip 5pt
{\bf (3)} If $\mu_\mathfrak a\neq 0$, then by Lemma~\ref{l:5.2}, $\max(q^l, q^{e-l})<q^{e_0}$ and  
	$$\mu_\mathfrak a = \frac{r}{\min(q^{e_0-l}, q^{e_0-e+l})} 
	,$$
	for some integer $r$ with $1\leq r<\min(q^{e_0-l},q^{e_0-e+l}).$
	Similiarly 
		$$\mu_\fb = \frac{r'}{\min(q^{e_0-l}, q^{e_0-e+l})} 
	,$$
	for some integer $r'$ with $1\leq r'<\min(q^{e_0-l},q^{e_0-e+l}).$
Since 
	$$\epsilon M +\frac{r}{\min(q^{e_0-l}, q^{e_0-e+l})} = \epsilon m_\mathfrak a p_1^{m_1}\cdots p_n^{m_n}q^m=\epsilon M_0 \max(q^{e-2l}, 1)\cdot q^m$$ 
	it follows that $\epsilon M\min(q^{e_0-l}, q^{e_0-e+l})+r = \epsilon M_0 q^{m+e_0-l}$. This implies that  $m=-e_0+l$ and $\epsilon M_0\equiv r\pmod{\min(q^{e_0-l}, q^{e_0-e+l})}$. 
	Since $\epsilon M_0 \mu_\fb -\mu_\fa \in \Z,$ we deduce that 
	$\epsilon M_0 r' \equiv r \pmod{\min(q^{e_0-l},q^{e_0-e+l})},$ and hence that $r'=1.$
	It follows that 
	$\mu_\fb =\min(q^{e_0-l}, q^{e_0-e+l})^{-1}=m_\fb q^m.$
	
\vskip 5pt
$\bullet$ $l\leq e-l$ and $\mu_{\mathfrak a}=r\mu_{\mathfrak a_0}\neq 0$: In this case  it follows from \eqref{e:5.4} 
and the definitions of $c'$ and $j$ that
	$$A\left(\frac{1}{cq^l}, \epsilon M\right) = A(\infty, \epsilon)\prod_{i=1}^n A(\infty, p_i^{m_i})\left(A\left(\frac{1}{c'q^l}, 0\right)e_\infty\left(q^{-e_0+l}\cdot j\right)\chi\left(jc'q^l+1\right)^{-1}\right),$$
	where $c'$ and $j$ are determined by $c$ and $\epsilon M_0$  as above.
	
\vskip 5pt
$\bullet$ $l>e-l$ and $\mu_\mathfrak a=r\mu_{\mathfrak a_0}\neq0$: In this case it follows from \eqref{e:5.3} and the definition of $c'$ that
$$A\left(\frac{1}{cq^l}, \epsilon M\right) = A(\infty, \epsilon)\prod_{i=1}^n A\left(\infty, p_i^{m_i}\right)A\left(\frac{1}{c'q^l}, 0\right),$$
	where  $c'$ is determined by $c$ and $\epsilon M_0$ as above.
	
\vskip 10pt
{\bf (4)} For a fixed integer $1\leq l<e$, take an integer $1\leq c<\min (q^l, q^{e-l})$ and $(c, q)=1$. Let $\mathfrak a=\frac{1}{cq^l}$. 
As shown above, $\mathfrak b = \frac{1}{c'q^l},$ where $c'$ is the unique solution to 
$$1\leq c'<\min(q^l, q^{e-l})\qquad c'\epsilon M_0-c\equiv 0\pmod{\min(q^l, q^{e-l})}.$$
Assume that $\mu_\mathfrak a=0.$ Then $\mu_\mathfrak b=0$ by Lemma~\ref{l:5.2}.  Furthermore $m_\fb =m_\fa
=q^{\max(e-2l,0)}.$ It follows that $q^mm_\fb -\mu_\fb = q^\alpha$ (where $\alpha$ is the highest pwer of $q$ that divides $M$ as before), which is integral.
\vskip 5pt
$\bullet$ $l\leq e-l$ and $\mu_{\mathfrak a}=0$: In this case it follows from \eqref{e:5.4} and the definitions 
of $c'$ and $j$ that 
	$$A\left(\frac{1}{cq^l}, \epsilon M\right) = A(\infty, \epsilon)\prod_{i=1}^n A(\infty, p_i^{m_i}) \left(A(\frac{1}{c'q^l}, q^{e-2l+m})e_\infty(q^m \cdot j)\chi(jc'q^l+1)^{-1}\right).$$
	
\vskip 5pt
 $\bullet$ $l>e-l$ and $\mu_\mathfrak a=0$: In this case it follows from \eqref{e:5.3} and the definition of $c'$ that
	$$A\left(\frac{1}{cq^l}, \; \epsilon M\right) = A(\infty, \epsilon)\prod_{i=1}^n A(\infty, p_i^{m_i})\cdot A\left(\frac{1}{c'q^l}, \; q^m\right).$$
\end{proof}

\begin{remark}
It is clear from the proofs that theorems
\ref{t:3.8} and \ref{t:main} are valid not only for 
Maass-Hecke newforms, but whenever the 
factorization $W_f =\prod_vW_{f,v}$ is valid.
\end{remark}

\section{Remarks on choices of $\fa$ and $\gamma_\fa$}\label{s:RemarksOnChoices}
As remarked in 
section~\ref{s:Intro}, the Fourier coefficients $A( \fa, n)$ of a Maass form 
$f$ at a cusp $\fa$ actually depend on the matrix $\sigma_\fa,$ or, equivalently, 
the matrix $\gamma_\fa$ used in its definition, and not only on the choice of $\fa.$
Further, while it is intuitively obvious that when considering Fourier 
expansions at various cusps, it is sufficient to consider a maximal set of $\Gamma_0(N)$-inequivalent 
cusps, it is also clear that the choice of representative for each $\Gamma_0(N)$-equivalence
class will influence the precise numbers considered.
In this section we make these dependencies completely explicit and then offer some
remarks on choice of representatives for $N$ not a prime power.

Because we wish to study the dependence of the Fourier coefficients
on the choice of matrix $\gamma_\fa$ used to define them, it is 
necessary to make this dependence explicit.  Thus, we write $A(\gamma_\fa, n)$
rather than $A( \fa, n).$  
\begin{lemma}\label{l:6.1}
Suppose that $\mathfrak a$ and $\mathfrak a'$ are two $\Gamma_0(N)$-equivalent 
cusps, and that $\gamma_\mathfrak a, \gamma_{\mathfrak a'}$ are two 
elements of $SL(2, \Z)$ such that $\gamma_\fa \infty = \fa$ and 
$\gamma_{\fa'} \infty = \fa'.$  Let $A(\gamma_\fa, n)$ (resp. $A(\gamma_{\fa'}, n)$), $n \in \Z$ 
 denote 
the Fourier coefficients of a Maass form $f$ at $\fa$ (resp. $\fa'$) defined
using an element $\sigma_\fa$ (resp. $\sigma_{\fa'}$) obtained from 
$\gamma_\fa$ (resp. $\gamma_{\fa'}$) as in 
section~\ref{s:Intro}. 
Then 
$$A(\gamma_{\fa'}, n) = \widetilde\chi(\gamma_0) \cdot e\left( (n+\mu_\fa)\cdot
\frac{j}{\fm_\fa}\right) \cdot A(\gamma_\fa, n),$$
where $\gamma_0 \in \Gamma_0(N)$ and $j \in \Z$ with $0 \le j < m_\fa$ are uniquely determined by the 
condition that 
$$\gamma_{\fa'} = \gamma_0 \cdot \gamma_\fa \cdot \begin{pmatrix} 1&j \\ 0 & 1 \end{pmatrix}.$$
\end{lemma}
\begin{proof}
This follows easily from the definitions.
\end{proof}
\vskip 10pt

We would like to extend the idea for choosing explicit representatives for the 
equivalence classes of cusps described in 
Lemma~\ref{l:4.1}.
It is not convenient or necessary to make a  completely explicit, 
choice of cusp representatives.  
It turns out to be sufficient to specify our representatives only modulo a suitable 
power of each prime dividing $N.$ 

For the remainder of 
 this section and the next, 
 we shall employ the following notation.  We take $S$ to be a finite set 
of primes, denoting a general element of $S$ by $q,$ and a general prime 
which is not in $S$ by $p.$  For each element $q$ of $S$ we fix a strictly 
positive integer $e_q,$ and we let $N = \prod_{q \in S } q^{e_q}.$

\begin{lemma}\label{l:6.2}
For each $q\in S,$ let $\pi_q : \Z/N\Z \to  \Z / q^{e_q} \Z$ denote the 
natural projection.  
The natural map $\Z/N\Z \to \prod_{q \in S} \Z / q^{e_q} \Z$ 
given by $n \mapsto (\pi_q(n) ) _{q\in S}$ 
induces a bijection 
$\P^1(\Z/N\Z) \to \prod_{q\in S} \P^1( \Z/q^{e_q}\Z).$ 
Furthermore, two elements of $\P^1(\Z/N\Z)$ are in the same 
$\Gamma_\infty$-orbit if and only if their images in  $\P^1( \Z/q^{e_q}\Z)$
are in the same $\Gamma_\i$-orbit for all $q\in S.$
\end{lemma}
\begin{proof}
The Chinese remainder theorem states that the natural map $\Z/N\Z \to \prod_{q \in S} \Z / q^{e_q} \Z$ 
is a ring isomorphism.  It follows easily that gives  a bijection 
 $$
 \Big\{ (x_0, x_1) \in (\Z/N\Z)^2 \;\Big |\; \langle x_0, x_1 \rangle  = \Z/N\Z\Big\}
 \to \prod_{q\in S}
 \Big\{ (x_0, x_1) \in (\Z/q^{e_q}\Z)^2 \;\Big |\; \langle x_0, x_1 \rangle  = \Z/q^{e_q}\Z\Big\}.
 $$
 Furthermore, if $(x_0', x_1') = (\lambda x_0, \lambda x_1),$  then $ 
 (\pi_q(x_0'), \pi_q(x_1')) = ( \pi_q(\lambda) \pi_q( x_0),  \pi_q(\lambda) \pi_q( x_1))
 $
for each $q\in S.$  Finally,  suppose   for each $q\in S$ there exists 
$\lambda_q$ such that 
 $(\pi_q(x_0'), \pi_q(x_1')) = (\lambda_q \pi_q( x_0),  \lambda_q \pi_q( x_1)).$
 Then it follows that $(x_0', x_1') = (\lambda x_0, \lambda x_1),$ where 
 $\lambda$ is the unique solution to the system of congruences
 $\pi_q(\lambda ) = \lambda_q\; \forall q\in S.$
Consequently,  we have a well-defined bijection 
$$\P^1(\Z/N\Z) \to \prod_{q\in S} \P^1( \Z/q^{e_q}\Z).$$
In the same manner, we see that 
$\exists n \in \Z/N\Z$ such that 
$[x_0:x_1] = [y_0:y_1]\left(\begin{smallmatrix} 1&n\\0&1\esm$
if and only if, for each $q,$ 
$\exists n_q \in \Z/q^{e_q}\Z$ such that 
$[x_0:x_1] \equiv  [y_0:y_1]\left(\begin{smallmatrix} 1&n_q\\0&1\esm\pmod{q^{e_q}}.$
\end{proof}

\begin{corollary}\label{c:6.3}
Suppose that, for each $q \in S,$ a set 
$\mathcal C_q$ of representatives for the double cosets 
$\Gamma_0(q^{e_q}) \bs SL(2, \Z) / \Gamma_\infty$ has been chosen.
Let $\mathcal C$ be a set having the property that, for any element
$(\gamma_q)_{q\in S}$ of the Cartesian product 
$\prod_{q\in S}\mathcal C_q$ there is a unique element $\gamma \in \mathcal C$
such that 
$$\gamma \;\equiv\; \gamma_q \hskip -3pt\pmod{q^{e_q}}, \qquad (\forall q\in S).$$
Then $\mathcal C$ is a set of representatives for the double cosets 
$\Gamma_0(N) \bs SL(2, \Z) / \Gamma_\infty.$
\end{corollary}
\vskip 10pt \noindent 
{\bf Remark:} 
A choice of representatives for $\Gamma_0(N) \bs SL(2, \Z) / \Gamma_\infty$
is slightly more information than a choice of representatives for the 
$\Gamma_0(N)$-equivalence classes of cusps:  it includes 
also a choice of matrix $\gamma_\mathfrak a$ for each
representative  cusp $\mathfrak a.$    
\vskip 10pt\noindent 
By Corollary~\ref{c:6.3}, we may fix a set $\mathcal C$ of representatives for the 
 double cosets 
$\Gamma_0(N) \bs SL(2, \Z) / \Gamma_\infty$
such that, for each $q \in S$ and each $\gamma\in \mathcal C,$ 
the matrix $\Gamma$ is equivalent $\pmod{q^{e_q}}$ to one of the 
representatives for $\Gamma_0(q^{e_q}) \bs SL(2, \Z) / \Gamma_\infty$
fixed in 
Lemma~\ref{l:4.1}:
\begin{equation}\label{e:DoubleCosetReps}
\left\{
\bpm 1&0\\0&1 \ebpm , \bpm 0&-1\\1&0 \ebpm 
\right\} \; \cup \; 
\left\{
\bpm 1&0\\ c_1q^l & 1\ebpm \left|
\gathered 
1 \le l < e, \ 
	\gcd(c_1,q) = 1,
	\\
	1 \le c_1 < \min(q^l,q^{e-l})
\endgathered \right.
\right\}.
\end{equation}
Such a choice determines a maximal set of inequivalent cusps for $\Gamma_0(N)$ 
and a choice of matrix $\gamma_\fa$ for each element $\fa$
of this set. Declaring that we choose our representatives 
$\fa$ and the corresponding matrices $\gamma_\fa$ in this manner does not uniquely 
determine $\gamma_\fa,$  but it does uniquely determine the coefficients
$A( \gamma_\fa, n),$ for if $\gamma_\fa$ and $\gamma_\fa'$ are equivalent 
$\pmod{N}$ then they differ by an element of the principal congruence subgroup 
$\Gamma(N)$ on the left, and $\Gamma(N)$ is contained in the kernel 
of the character $\widetilde \chi.$ 
\vskip 10pt 
\begin{lemma}\label{l:6.5}
Let $\fa$ be a cusp and $\gamma_\fa$ 
a matrix such that $\gamma_\fa\infty = \fa$ and, for each $q \in S,$ 
$\gamma_\fa$ is equivalent $\pmod q^{e_q}$ to one of the elements of \eqref{e:DoubleCosetReps}. 
Let $\fa_q$ denote the corresponding cusp.  
That is, $\gamma_{\fa_q} \infty = \fa_q$ with $\gamma_{\fa_q}$ from \eqref{e:DoubleCosetReps} 
and $\gamma_\fa \equiv \gamma_{\fa_q} \pmod{q^{e_q}}.$
Let $\mu_\fa$ be the cusp parameter of $\fa,$ defined using some character $\chi
\pmod{N}.$  The isomorphism $\Z/N\Z \to \prod_{q\in S} \Z/q^{e_q} \Z$ ensures 
that $\chi = \prod_{q \in S} \chi_q$ for some characters $(\chi_q)_{q\in S}$ with 
$\chi_q \pmod{q^{e_q}}$ for each $q.$   For each $q\in S,$ let 
$\mu_{\fa_q}$ denote the cusp parameter of $\fa_q$ relative to 
$\chi_q.$  
Then 
$$ m_\fa \; = \; \underset {q\in S} {\operatorname{lcm}} \; ( m_{\fa_q}),$$
$$\mu_\fa \; = \; \sum_{q \in S}
\frac{m_\fa}{m_{\fa_q}}
 \mu_{\fa_q} - \left\lfloor \sum_{q \in S} 
 \frac{m_\fa}{m_{\fa_q}}
 \mu_{\fa_q} \right\rfloor \quad ( \text{greatest integer function}).$$
\end{lemma}
\begin{proof}
The lower left entry of $\gamma_\fa \cdot \left(\begin{smallmatrix} 1&j \\ 0 & 1 \esm \cdot \gamma_{\fa}^{-1}$
is congruent to $0 \pmod{N}$ if and only if it is congruent to $0 \pmod{q^{e_q}}$ for 
each $q.$  This is the case if and only if $j$ is divisible by $m_{\fa_q}$ for each $q.$  
The first statement follows.  We see at once that  for each $q \in S,$
$$\gamma_\fa \cdot 
\bpm 1& m_\fa \\ 0& 1 
\ebpm 
\cdot \gamma_\fa^{-1} 
\equiv 
 \gamma_{\fa_q} \cdot 
\bpm 1& m_\fa \\ 0& 1 
\ebpm 
\cdot \gamma_{\fa_q}^{-1} 
= 
 \gamma_{\fa_q} \cdot 
\bpm 1& m_{\fa_q} \\ 0& 1 
\ebpm ^{ \frac{m_\fa}{m_{\fa_q}}}
\cdot \gamma_{\fa_q}^{-1} 
\pmod{q^{e_q}}
.$$
It follows that $d_{\fa} \equiv d_{\fa_q} ^{ \frac{m_\fa}{m_{\fa_q}}}\pmod{q^{e_q}}$
for all $q \in S$ and from this the second assertion follows immediately.
\end{proof}
    \vskip 10pt
The following lemma will be useful later on.    
\begin{lemma}\label{l:6.6}
Let $\fa$ be a cusp and let $(c , d)$ denote the bottom row of $\gamma_\fa.$  
Let $a$ be an integer prime to $N.$  Let $\fa'$ be the cusp such that $\gamma_{\fa'}$ 
represents the double coset in $\Gamma_0(N)\bs SL(2, \Z)/ \Gamma_\i$ corresponding
to the $\Gamma_\i$-orbit of $[ac :d]$ in $\P^1( \Z/N\Z).$  Then 
$m_{\fa'} = m_\fa$ and $\mu_{\fa'} - a \cdot \mu_\fa \in \Z.$ 
\end{lemma}
\begin{proof}
For each $q$ in $S,$ the pair $(c,d)$ is equivalent to $(0,1), (1,0)$ or
$(c_1 q^l, 1)$ modulo $q^{e_q}$ where $c_1, l$ are subject to the constraints in \eqref{e:DoubleCosetReps}. 
It follows at once that the bottom row of $\gamma_{\fa'}$ is equivalent 
to $(0,1),(1,0)$ of $(c_1' q^l , 1),$ respectively, where $c_1'  \equiv ac_1 \pmod{q^{e_q-l}}$ 
and $0 < c_1' < q^{\min(l, e_q-l)}.$ 
The value of $m_{\fa_q}$ is $1,$ except  when  the bottom row of $\gamma_{\fa}$ is equivalent 
to  $(c_1 q^l , 1),$ with $l < e-l,$ in which case it 
is $q^{e-2l}.$ It follows easily that $m_{\fa'_q} = m_{\fa_q}$ for each $q$ and thence that
$m_{\fa'}=m_\fa.$  Similarly, $d_{\fa_q}$ is equal to $1$ if $(c,d) \equiv (0,1)$ or  $(1,0)
\pmod{q^{e_q}},$ while if it is congruent to $(c_1 q^l, 1),$  then $d_{\fa_q} = (1+c_1q^{\max( l, e-l)}).$
Clearly, with $c_1'$ as above, we have
$$( 1+ c_1' q^{\max( l, e-l)}) \;
\equiv \; ( 1+ ac_1q^{\max( l, e-l)}) 
\; \equiv \; ( 1+ c_1' q^{\max( l, e-l)})^a \pmod{q^{e_q}}.$$
It follows that $a\mu_{\fa_q}  - \mu_{\fa'_q} \in \Z.$  
The second assertion of this lemma now follows from Lemma~\ref{l:6.5}, because 
$m_\fa/m_{\fa_q}$ and $m_{\fa'}/ m_{\fa'_q}$ are the same integer.
\end{proof}
    \vskip 10pt 
    It is easy to see from the proof of Lemma~\ref{l:6.6} that the mapping 
    $(a, \mathfrak a) \to \mathfrak a'$ which is considered in Lemma~\ref{l:6.6} actually defines
    an action of $(\mathbb Z/ N\Z)^\times$ on our set of cusps.   Abusing notation we regard it
    as an ``action'' of the set of all elements of $\Z$ which are prime to $N.$   We shall write 
    $a \cdot \fa$ for the cusp $\fa'$ obtained from $\fa$ in this fashion, and 
    $a^{-1} \cdot \fa$ for the unique cusp $\fa''$ such that 
    $a \cdot  \fa'' = \fa.$

\section{On Stark's Question}
In this section we deduce some consequences of theorem \ref{t:3.8}
which provide a partial answer to the question of Stark posed in the introduction.  
We first give in theorem \ref{t:suffCond}
a sufficient condition for 
multiplicative relations at a cusp, and deduce that there are multiplicative relations at every cusp if $N$ is equal to $4$ times a squarefree odd number.  Next, if $N$ is equal to $8$ times a squarefree odd number, then theorem \ref{t:suffCond} will imply multiplicative relations at all cusps except for those of the form $a/b$ with $2\mid \mid b.$  In proposition \ref{p:eight} we consider such cusps in detail.   

In order to begin, it is useful to introduce an alternate notation for the Fourier coefficients.
Define 
\begin{equation}\label{e:bNotation}
B( \fa, \alpha) := 
 \begin{cases}
A(\fa, m_\fa \alpha - \mu_\fa),& \text{ if } m_\fa \alpha- \mu_\fa \in \Z,\\
0, & \text{ otherwise,}
\end{cases}
\qquad (\alpha \in \Q^\times).
\end{equation}
When the dependence of $B(\fa, \alpha)$ on the choice of $\gm_\fa \in SL(2, \Z)$
is relevant, we shall denote $B(\fa, \alpha)$
by $B(\gm_\fa, \alpha),$ 
where the coefficients $A( \fa, \cdot)$ 
are defined using $\sigma_\fa = \gm_\fa \bspm \sqrt{m_\fa}&\\& \sqrt{m_\fa}^{-1}\espm.$
Note that this defines $B( \gm, \alpha)$ for 
any $\gm \in SL(2, \Z)$ and any $\alpha \in \Q^\times,$ and permits us to consider, where necessary coefficients defined at the same cusp relative to two different elements of $SL(2, \Z).$

Next, we give a formal definition of what it means for a function $\Q^\times \to \C$ 
to be multiplicative.

\begin{definition}
Let $h: \Q^\times \to \C$ be a function. 
Then $h$ is said to be multiplicative if there 
exist functions $h_\infty: \{ \pm 1\} \to \C$
and $h_p:\Z \to \C$ for each prime $p$ 
such that $h_p(0) =1$ for all but finitely 
many primes $p,$ and 
\begin{equation}\label{e:multiplicativity}
h\left( \epsilon \prod_p p^{e_p}\right) = 
h_\infty( \epsilon) \cdot \prod_p h_p(e_p),
\end{equation}
for any $\epsilon \in \{ \pm 1\}$ and 
$e_p \in \Z$ with $e_p=0$ for all but finitely many primes $p.$
\end{definition}
\begin{remark}
The functions $h_\infty, h_p$ are not uniquely determined by $h$:  given a collection of functions satisfying \eqref{e:multiplicativity}, one may vary any finite set of them by nonzero scalars, provided the product of all the scalars is $1.$  
If $h$ is a factorizable function and $h(1)$
is $1,$ then one can normalize by requiring that $h_\infty(1)=1$ and $h_p(0)=1$ for all $p.$  Then \eqref{e:multiplicativity} becomes
\begin{equation}\label{e:simpleMult}
h\left( \epsilon \prod_p p^{e_p}\right) = 
h( \epsilon) \cdot \prod_p h\left(p^{e_p}\right).
\end{equation} 
Thus \eqref{e:multiplicativity} is a generalization of \eqref{e:simpleMult}, which may be applied to functions with vanish at $1.$ 
\end{remark}
In this section, we consider the question:
\begin{equation}\label{question}
\text{\it is the function 
$\alpha \to B( \gamma, \alpha)$ multiplicative?}\end{equation}

We remark that by lemma \ref{l:6.1}, the 
answer to the question \eqref{question}
depends only on the cusp $\fa$
 provided that $\mu_\fa =0$ and $m_\fa =1,$ but not in general.

\begin{theorem}\label{t:suffCond}
Let $\fa=\frac{a}{b}$ be a cusp, with $\gcd(a,b)=1,$ and set $M = \gcd(N,b).$
 Assume that either $\gcd(M,N/M)=1,$ 
 or else $\gcd(M,N/M)=2,$  and $2 \mid \mid \frac NM.$  
 Assume further that for each prime $q\mid N,$ the matrix 
 $\gamma_\fa$ is equivalent to some element of 
 \eqref{e:DoubleCosetReps} modulo 
 $q^e,$ where $q^e \mid\mid N.$ 
 Finally, let $f:\mathfrak h \to \C$ be any Maass form for $\Gamma_0(N)$ such that 
 $W_f(g) = \prod_v W_{f,v}(g_v)$ for all 
 $g = \{ g_v\} \in GL(2, \A).$   
 Then the Fourier coefficients $B(\gm_\fa, \alpha)$ of $f$ at $\fa$ (defined relative to $\gm_\fa$) are multiplicative.
 \end{theorem}
 \begin{proof}
  For $\gm \in SL(2, \Z)$ and $\alpha \in \Q^\times,$ let  
 \begin{equation}\label{gaalpha}
 g( \gm, \alpha) = 
i_\infty\left( \bpm \text{sign}(\alpha)&\\&1 \ebpm
\right)
 i_{_{\text{finite}}}\left( \bpm \alpha &\\&1\ebpm \gamma^{-1}
\right).
 \end{equation}
 Then it follows from theorems 
 \ref{t:GlobalWhittaker} and  \ref{t:3.8}
 that 
 \begin{equation}\label{e:baalpha=prod}
 \begin{aligned}
 B(\gm, \alpha)&=
 \lim_{y_\infty \to 1}\; \frac{W_f\left( i_\infty\left(\begin{pmatrix} y_\infty & \\ & 1\end{pmatrix}\right)\cdot g( \gm, \alpha) \right)}{ W_{\frac{\text{sign}(\alpha)k}2, \nu -\frac 12} (4\pi y_\infty)}\\
  & =
  A(\infty, \text{sign}(\alpha))
  \cdot \prod_p W_{f,p}\left(\begin{pmatrix} \alpha & \\ & 1\end{pmatrix}\gamma^{-1}\right).
  \end{aligned}
 \end{equation}
  Furthermore, $$
 W_{f,p}\left(\begin{pmatrix} \alpha & \\ & 1\end{pmatrix}\gamma^{-1}\right)
 =\begin{cases}
 A( \infty , |\alpha|_p^{-1}), &\text{ if } |\alpha|_p^{-1} \in \Z, \\0&, \text{otherwise,}
 \end{cases}
 $$
 for all $p\nmid N.$  Let $q$ be a prime with $q^e \mid \mid N,\, e>0.$ 
 In order compute $W_{f,q}\left(\begin{pmatrix} \alpha & \\ & 1\end{pmatrix}\gamma_{\fa}^{-1}\right)$
  one must determine the 
  unique cusp
 $\fb( \fa, \alpha, q),$ 
 (from our fixed set of representatives for the $\Gamma_0(N)$-equivalence classes),
 integer  $0\le j( \fa, \alpha, q)\le m_\fa,$ and matrix  
 and $k_0( \fa, \alpha, q)\in K_0(N)$
 such that
 \begin{equation}\label{e:b,j,and k_0}
 i_{_{\text{finite}}} \left( 
 \gm_{\fb( \fa, \alpha, q)} \bpm 1&j( \fa, \alpha, q)\\&1 \ebpm 
 \right)
 i_q\left( \bpm \alpha |\alpha|_q&\\&1 \ebpm 
 \gm_\fa^{-1}
 \right) =k_0( \fa, \alpha, q).
 \end{equation}  We may assume that the set of representatives for the $\Gamma_0(N)$-equivalence classes, and that a 
 matrix $\gamma_\fb$ for each representative $\fb$, were chosen as in section \ref{s:RemarksOnChoices}.  Then 
 $\fb(\fa , \alpha, q)$ depends only on the reduction of 
 $\fa$ modulo $q^e.$  Furthermore, 
 this reduction has to equal $\bspm 1&0\\0&1 \espm$ or $\bspm 0&-1\\1&0 \espm,$ unless $q=2,$ in which case $\bspm 1&\\2^{e-1}&1 \espm$ is also an option.  
For such matrices the values of $\fb, j$ and $k_0$ are as below:
$$
\begin{array}{|r|r|r|r|}\hline
\gamma_\fa&\gamma_\fb & j&k_0\\
\hline
\bpm 1&\\&1 \ebpm& \bpm 1&\\&1 \ebpm&0&
\bpm \alpha|\alpha|_q &\\&1 \ebpm\\
\bpm 0&-1\\1&0 \ebpm &\bpm 0&-1\\1&0 \ebpm & 0 & \bpm1&\\& \alpha|\alpha|_q \ebpm\\
\bpm 1&\\2^{e-1}&1\ebpm &
\bpm 1&\\2^{e-1}&1\ebpm & 0 &  
\bpm 1&0\\ 2^{e-1} (\alpha^{-1} |\alpha|^{-1}_2-1)&1 \ebpm
\bpm \alpha|\alpha|_2 &\\&1 \ebpm\\ \hline
\end{array}
$$
Multiplicativity of $B(\gm_\fa, \alpha)$ follows easily.
   \end{proof}
   \begin{corollary}
   If $N$ is either squarefree, or equal to $4$ times an odd, squarefree number, then 
   $B(\fa, \alpha)$ is multiplicative at every cusp, provided the matrices $\gamma_\fa$ are chosen as in section \ref{s:RemarksOnChoices}.
   \end{corollary}
   \begin{remark}
 If we do not choose the matrices $\gamma_\fa$ as in section \ref{s:RemarksOnChoices}, then the multiplicativity may well be destroyed.  For example, suppose $N=q,$ a prime, and $\fa =0.$   Then 
 $\mu_0 = 0,$ and $m_0 = q.$  Thus $B(0, \alpha)$ is zero, unless $\alpha = n/q$ with $n \in \Z,$ in which case it is $A(0,n).$  
 By choosing $\gamma_0 = \bspm 0&-1\\1&0 \espm,$ we obtain the Fricke involution for $\sigma_0,$ and the Fourier coefficents $A(0,n)$ and $B(0, \alpha)$ are multiplicative
exactly as in \cite{Asai:1976}
 On the other hand, if we take
 $\gamma_0 = \bspm 0&-1\\1&1 \espm,$ then, according to Lemma \ref{l:6.1}, the effect is to multiply $A(0, n)$ by $e_\infty(\frac{n}q),$ destroying multiplicativity. 
    \end{remark}
   \begin{remark}
   The condition placed on $M$ in theorem \ref{t:suffCond} can be given more conceptually.  What we require is a condition on $\fa$ which will ensure that 
   $\fb$ is independent of $\alpha,$ and $j$ is always zero.  Next, we examine a case 
   when $\fb$ is independent of $\alpha,$
   but $j$ is not always zero.
   \end{remark}
   \begin{proposition}\label{p:eight}
   Suppose that $N$ is  equal to $8$ times an odd squarefree number, and that $\fa= a/b$ is a cusp such that $2\mid\mid b.$  Assume that $\gamma_\fa$ was chosen as in section \ref{s:RemarksOnChoices}.  Then 
   $B(\fa, \alpha)$ is equal to a multiplicative function times the function 
  \begin{equation}\label{e:nonMultFcnForEight}
   \alpha = 2^e \cdot(1+2j+4k) \mapsto e_\infty( 2^e \cdot j).
   \end{equation}
   \end{proposition}
   \begin{proof}
   As in the proof of theorem \ref{t:suffCond}, 
   we use \eqref{e:baalpha=prod}.  
   The conditions on $N$ assure that
   $B( \fa, \alpha, q)$ is independent 
   of $\alpha$ for all $q,$ and that 
   $j(\fa, \alpha, q)$ is identically zero for all $q$ except $q=2.$  When $q=2,$ we must solve   $$
 i_{_{\text{finite}}} \left( 
 \gm_{\fb} \bpm 1&j\\&1 \ebpm 
 \right)
 i_2\left( \bpm \alpha |\alpha|_2&\\&1 \ebpm 
 \gm_\fa^{-1}
 \right) \in K_0(N).
 $$
By looking at the condition on $j$ at 
each prime, we see that $\frac N8 \mid j,$ and that 
$$
\bpm 1&\\2&1 \ebpm \bpm 1&j\\&1 \ebpm 
\bpm \alpha |\alpha|_q&\\&1 \ebpm
\bpm 1&\\-2&1 \ebpm
=\bpm \alpha|\alpha|_2-2 j&j\\2 \alpha|\alpha|_2-2 (1+2 j)&1+2 j\ebpm
 \in K_0(8).
$$
 This implies that $j$ satisfies $\alpha|\alpha|_2 \equiv 1+2j \pmod{4}.$  
Since $\{ 1, 3\}$ is a subgroup of 
$\Z/8\Z^\times,$ the function
$$ \alpha \mapsto \widetilde \chi_{_{\text{idelic}}}\bpm \alpha|\alpha|_2-2 j&j\\2 \alpha|\alpha|_2-2 (1+2 j)&1+2 j\ebpm,$$
is multiplicative. 
It follows that  $B(\fa, \alpha)$ is 
a product of functions, all of which are
multiplicative except for 
$\alpha \mapsto e_\infty( j( \fa, \alpha, 2) \cdot 2^e),$  
which is given by \eqref{e:nonMultFcnForEight}.
   \end{proof}

Certainly, the function \eqref{e:nonMultFcnForEight} is not multiplicative. 
However, it is still possible for the Fourier coefficients $B(\fa, \alpha)$ to be multiplicative if sufficiently many of them are zero.
For example, suppose that $N=8$ and  $\chi$ is primitive.  Then 
$\mu_\fa =\frac 12,$ and $m_\fa = \frac 12,$ 
from which it follows that $B(\fa, \alpha) =0$
unless $\alpha= \frac{2n+1}{4}.$ 
The function 
\eqref{e:nonMultFcnForEight} 
in this case is $1$ if $2n+1$ is $1$ mod $4$
and $i$ if it is $3$ mod $4.$ 
 If there exist coprime integers $n_1, n_2$ with 
$n_1 \equiv n_2 \equiv 3 \pmod{4},$ 
and $B( \fa, \frac{n_1}4)B( \fa, \frac{n_2}4) \ne 0,$ then it follows that $B(\fa, \alpha)$ is not a multiplicative function.  
Furthermore, it 
follows from lemma \ref{l:6.1}
that $B(\fa, \alpha)$ is not multiplicative for any other choice of $\gamma_\fa$ either.  

On the other hand, if $A(\infty, n)=0$
whenever $n\equiv 3 \pmod{4},$ then the function 
 \eqref{e:nonMultFcnForEight} can be 
 omitted from the formula for $B(\frac12, \alpha)$ without affecting its value.  As a consequence, the Fourier coefficients are 
 multiplicative in this case.

\section{A classical interpretation of the Jacquet-Langlands criterion for supercuspidality}\label{s:Supercuspidal}

	The following theorem provides a classical criterion  for a representation of $GL(2, \mathbb Q_p)$ to be supercuspidal.
	
	\begin{theorem}\label{t:3.9}
Fix a prime $p$, let $V_p$ be a complex vector space,  and let $\pi_p:GL(2, \mathbb Q_p)\to GL(V_p)$. Assume  $(\pi_p, V_p)$ is an irreducible and admissible representation of $GL(2, \mathbb Q_p)$.

Let $f$ be a normalized Maass-Hecke newform of weight $k$, type $\nu$, level $N$,  and  character $\chi\pmod{N}$ for $\Gamma_0(N)$.  For each cusp $\mathfrak a\in \mathbb Q\cup\{\infty\}$ and $n\in\mathbb Z$, let $A(\mathfrak a, n)$ denote the $n^{th}$ Fourier coefficient of $f$ at the cusp $\mathfrak a$.  
 Let $f_{_{\text{adelic}}}$ be the adelic lift of $f$  as in \eqref{e:AdelicLift}. If $(\pi_p, V_p)$ is isomorphic to the local representation factor of the irreducible global automorphic representation of $GL(2, \mathbb A)$ (which is generated by $f_{adelic}$), then $(\pi_p, V_p)$ is supercuspidal if and only if 
 for each cusp $\mathfrak a\in \mathbb Q\cup\{\infty\}$ with $\mu_\mathfrak a=0$, there exists an integer $M_\mathfrak a \ge 0$, such that 
	$$A(\mathfrak a, m_\mathfrak a p^m)=0, \qquad \big(\text{for all} \; m\in \mathbb Z, \,m\geq M_\mathfrak a\big),$$
where $m_\mathfrak a, \, \mu_\mathfrak a$ are given by \eqref{e:sigma_a}, \eqref{e:cuspParameter}, respectively.
\end{theorem}
\begin{proof}
	 If $p\nmid N$, then $f_{_{\text{adelic}}}$ is fixed by $K_p=GL(2, \mathbb Z_p)$. It follows that $(\pi_p, V_p)$ has a nonzero $K_p$-fixed vector, and a nonzero Whittaker model, which forces it to be an irreducible principal series representation (see \cite{Bump:1997}, Theorem 4.6.4). It follows that if $p \nmid N$ then $(\pi_p, V_p)$ cannot be a supercuspidal representation.  On the other hand, if $p\nmid N$ and 
	 $f$ is an eigenfunction of $T_p^\chi,$ then 
	 it follows easily that infinitely many 
	 of the coefficients $A(\infty, p^m)$ are nonzero.  This proves the equivalence in this case, and 
	 henceforth we shall assume that $p\mid N.$

Let $W_f$ denote the global Whittaker function of 
$f_{_{\text{adelic}}}$. It follows from Theorem~\ref{t:FactorizationOfWhittaker}, that there exist local Whittaker functions $W_{f, v}$ on $GL(2, \mathbb Q_v)$ at each place of $v$  such that 
	$$W_f(g) = \prod_v W_{f, v}(g_v), \qquad \big(\forall g = \{g_v\}_{v\le\infty} \in GL(2, \mathbb A)\big).$$
For a prime $p$, assume that $(\pi_p, V_p)$  is isomorphic to a component of the irreducible automorphic representation of $GL(2, \mathbb A)$ generated by $f_{_{\text{adelic}}}$. As we have noted before (see proof of Theorem~\ref{t:FactorizationOfWhittaker}) there exists a  Whittaker space $\mathcal W_p :=\mathcal W(\pi_p, e_p)$ associated to $(\pi_p, V_p)$ and	$W_{f, p}\in \mathcal W_p.$
We also have a corresponding Kirillov space, denoted $\mathcal K_p$, where 
	$$\mathcal K_p = \left\{ W\left(\begin{pmatrix} y & 0 \\ 0 & 1\end{pmatrix}\right)\;\bigg|\; y\in \mathbb Q_p^\times, \;\; W\in \mathcal W_p\right\}.$$
	(The motivation for this name 
	comes from   \cite{Kirillov:1963},  \cite{Kirillov:1966}.) 	
Define the Schwartz-Bruhat space
	$$\begin{aligned} S(\mathbb Q_p^\times) & :=  \left\{\phi:\mathbb Q_p^{\times} \to \mathbb C \; \left | \; \begin{matrix}
  \phi\, \text{is locally constant, and }
  \exists N_\phi > \epsilon_\phi> 0  \text{ such that }  \\ \text{$\phi(y) = 0$ if $|y|_p < \epsilon_\phi$ or $|y|_p > N_\phi$}\end{matrix}\right.\right\}.\end{aligned}$$
It was shown by Jacquet-Langlands 
that $(\pi_p, V_p)$ is supercuspidal if and only if $\mathcal K_p = S(\mathbb Q_p^\times)$. 
(This is proposition 2.16 of 
 \cite{Jacquet:1970}.  It also appears as theorem 3 of \cite{Godement:1970}.  See  
 \cite{GoldfeldHundley}, for a very elementary 
 treatment.)

Assume that $p\mid N$. For $y\in \mathbb Q_p^\times$, define
	$$\varphi_p(y):= W_{f, p}\left(\begin{pmatrix} y & 0 \\ 0 & 1\end{pmatrix}\right)\in \mathcal K_p.$$
	By Corollary~\ref{c:LocalWhittaker}
	$$\varphi_p(y)=\begin{cases} A(\infty, |y|_p^{-1}), & \hbox{ if } |y|_p^{-1}\in \mathbb Z,\\
	0, & \hbox{ otherwise.}\end{cases}$$
	Then $\varphi_p$ is invariant under the action of  $$I_{p, N} =  \left\{\begin{pmatrix} a & b\\ c& d\end{pmatrix} \in GL(2, \mathbb Z_p)\;\Big|\; c\in N\mathbb Z_p\right\},$$ and $\mathcal K_p$ is spanned by
	$$\Big\{\pi'(g).\varphi_p\;\Big |\; g\in GL(2, \mathbb Q_p)\Big\}.$$

Now fix $g=\left(\begin{smallmatrix} a & b\\ c & d\esm\in GL(2, \mathbb Q_p)$. We will compute $\pi'\left(g\right) \, . \, \varphi(y)$ for $y\in \mathbb Q_p^\times$ and determine under what conditions this function lies in $S\left(\mathbb Q_p^\times\right).$ There are two different cases that need to be considered.
\vskip 10pt
$\underline{\text{\bf Case (1)}\; c=0:} $ 	
	$$\begin{aligned} \pi'\left(\bpm a & b\\ 0 & d\ebpm\right).\varphi_p(y) & = W_{f, p}\left(\bpm y & 0 \\0 & 1\ebpm \bpm a & b\\ 0 & d\ebpm\right)\\
	& \\
	& =\begin{cases} \chi_p(d)\,e_p\left(bd^{-1}y\right)A\left(\infty, |ad^{-1}y|_p^{-1}\right), & \text{ if } |ad^{-1}y|_p^{-1}\in \mathbb Z,\\ 0, & \text{ otherwise.}\end{cases}\end{aligned}$$
For fixed $\left(\begin{smallmatrix} a & b\\ 0 & d\esm\in GL(2, \mathbb Q_p)$, the function $\pi'\left(\left(\begin{smallmatrix} a & b\\ 0 & d\esm\right) \, . \,\varphi_p\in S(\mathbb Q_p^\times)$ if and only if there exists\break an integer $M\ge 0$ such that $\pi'\left(\left(\begin{smallmatrix} a & b\\ 0 & d\esm\right).\varphi_p(y)=0$ if $y\in p^M \mathbb Z_p$. Since $\chi_p(d)\neq 0$ and $e_p(bd^{-1}y)\neq 0$, this function  vanishes if and only if $A\left(\infty, |ad^{-1}y|_p^{-1}\right)=0$. It follows that $
\pi'\left(\left(\begin{smallmatrix} a & b\\ 0 & d\esm\right)\, . \,\varphi_p\in S(\mathbb Q_p^\times)$ if and only if there exists a an integer $M_\infty\ge0$ such that $A(\infty, p^m)=0$ whenever  $m\geq M_\infty$. 

\vskip 10pt
$\underline{\text{\bf Case (2)}\;c\neq0:} $ 
	$$\bpm a & b\\ c& d\ebpm =\bpm c & 0 \\ 0 & c\ebpm \bpm c^{-2}(ad-bc) & ac^{-1} \\ 0 & 1\ebpm \bpm 0 & -1 \\ 1 & 0\ebpm \bpm 1 & c^{-1} d \\ 0 & 1\ebpm.$$
Consequently

	$$\begin{aligned} & \pi'_p\left(\bpm a & b\\ c& d\ebpm\right) \, . \,\varphi_p(y)\\
	& \\
	&\hskip 30pt = \pi'_p\left(\bpm c & 0 \\ 0 & c\ebpm\bpm c^{-2}(ad-bc) & ac^{-1} \\ 0 & 1\ebpm\right).\left(\pi'_p\left(\bpm 0 & -1 \\ 1 & 0 \ebpm \bpm 1 & c^{-1}d \\ 0 & 1\ebpm\right).\varphi_p\right)(y)\\
	& \\
	&\hskip 30pt =\chi_p(c)e_p\left(ac^{-1}y\right)\left(\pi'_p\left(\bpm 0 & -1 \\ 1 & 0 \ebpm \bpm 1 & c^{-1}d \\ 0 & 1\ebpm\right).\varphi_p\right)\big(c^{-2}(ad-bc)y\big).\end{aligned}$$

If $c^{-1}d\in\mathbb Z_p$ then 
	$$\left(\pi'_p\left(\bpm 0 & -1 \\ 1 & 0 \ebpm \bpm 1 & c^{-1}d \\ 0 & 1\ebpm\right).\varphi_p\right)(c^{-2}(ad-bc)y) = \left(\pi'_p\left(\bpm 0 & -1 \\ 1 & 0 \ebpm\right).\varphi_p\right)(c^{-2}(ad-bc)y)$$
	$$=W_{f, p}\left(\bpm c^{-2}(ad-bc) y & 0 \\ 0 & 1\ebpm \bpm 0 & -1 \\ 1& 0 \ebpm \right).$$
Then by Corollary~\ref{c:LocalWhittaker}, there exists a cusp $\mathfrak a\in \mathbb Q\cup\{\infty\}$, an integer $0\leq j<m_\mathfrak a$, and $k_{0}\in K_0(N)$, which are uniquely determined by $g=\left(\begin{smallmatrix} a & b\\ c& d\esm$ and $y$ such that 
	$$i_{_{\text{\rm finite}}}\left(\gamma_\mathfrak a\bpm 1 &  j \\ 0 & 1\ebpm\right)i_p\left(\bpm |c^{-2}(ad-bc)y|_pc^{-2}(ad-bc)y & 0 \\ 0 & 1\ebpm\bpm 0 & -1 \\ 1 & 0 \ebpm\right)=k_0\in K_0(N).$$
	Then
	$$\begin{aligned} & W_{f, p}\left(\bpm c^{-2}(ad-bc) y & 0 \\ 0 & 1\ebpm \bpm 0 & -1 \\ 1& 0 \ebpm \right)\\
	& \\
	& \hskip 50pt = \begin{cases} A\Big(\mathfrak a, m_\mathfrak a|c^{-2}(ad-bc)y|_p^{-1}-\mu_\mathfrak a\Big)\, e_\infty\Big(|c^{-2}(ad-bc)y|_p^{-1}j\Big)\,\widetilde\chi_{_{\text{\rm idelic}}}(k_0), \\
	\;\;\;\;\;\text{ if } m_\mathfrak a|c^{-2}(ad-bc)y|_p^{-1}-\mu_\mathfrak a\in\mathbb Z,\\
	0, \;\;\;\;\;\;\;\;\;\;\text{ otherwise.}\end{cases}\end{aligned}$$
Therefore, if $c^{-1}d\in\mathbb Z_p$ then
	$$\begin{aligned} & \pi'_p\left(\bpm a & b\\ c& d\ebpm\right).\varphi_p(y)\\
	& \\
	& \hskip 16pt=\begin{cases}\chi_p(c)e_p(ac^{-1}y)\, A\Big(\mathfrak a, m_\mathfrak a|c^{-2}(ad-bc)y|_p^{-1}-\mu_\mathfrak a\Big)e_\infty\Big(|c^{-2}(ad-bc)y|_p^{-1}j\Big)\widetilde\chi_{_{\text{\rm idelic}}}(k_0),\\
	\hskip 36pt\text{ if } m_\mathfrak a|c^{-2}(ad-bc)y|_p^{-1}-\mu_\mathfrak a\in\mathbb Z,\\
	0, \;\;\;\;\;\;\;\;\;\;\text{ otherwise.}\end{cases}\end{aligned}$$

If $c^{-1}d\notin \mathbb Z_p$, 
	$$\bpm 0 & -1 \\ 1  & 0 \ebpm \bpm 1 & c^{-1}d \\ 0 & 1\ebpm = \bpm c^{-1}d & 0 \\ 0 & c^{-1}d\ebpm \bpm c^2 d^{-2} & -cd^{-1}\\ 0 & 1\ebpm \bpm 1 & 0 \\ cd^{-1} & 1\ebpm$$
	and $\left(\begin{smallmatrix} 1 & 0 \\ cd^{-1} & 1\esm\in GL(2, \mathbb Z_p)$ since $cd^{-1}\in p\mathbb Z_p$. It follows that
	$$\begin{aligned} & \left(\pi'_p\left(\bpm 0 & -1 \\ 1 &0\ebpm \bpm 1 & c^{-1}d \\ 0 & 1\ebpm \right).\varphi_p\right)(c^{-2}(ad-bc)y)\\
	& \hskip 30pt = \chi_p\left(c^{-1}d\right)e_p\Big(-c^{-1}d^{-1}(ad-bc)y\Big)\left(\pi'_p\left(\bpm 1 & 0 \\ cd^{-1} & 1\ebpm\right).\varphi_p\right)\big(d^{-2}(ad-bc)y\big)\end{aligned}$$
	and
	$$\left(\pi'_p\left(\bpm 1 & 0 \\ cd^{-1} & 1\ebpm\right).\varphi_p\right)\big(d^{-2}(ad-bc)y\big)=W_{f, p}\left(\bpm d^{-2}(ad-bc)y & 0 \\ 0 & 1\ebpm \bpm 1 & 0 \\ cd^{-1} & 1\ebpm \right).$$
	There exists a cusp $\mathfrak a\in \mathbb Q\cup\{\infty\}$, an integer $0\leq j<m_\mathfrak a$ and $k_0\in K_0(N)$, which are uniquely determined by $g=\left(\begin{smallmatrix} a & b\\ c& d\esm$ and $y$ such that
	$$i_{_{\text{\rm finite}}}\left(\gamma_\mathfrak a \bpm 1 & j \\ 0 & 1\ebpm\right)i_p\left(\bpm |d^{-2}(ad-bc)y|_pd^{-2}(ad-bc)y & 0 \\ 0 & 1\ebpm \bpm 1 & 0 \\ cd^{-1} & 1\ebpm \right)=k_0\in K_0(N).$$
	It follows that
	$$\begin{aligned} & W_{f, p}\left(\bpm d^{-2}(ad-bc)y & 0 \\ 0 & 1\ebpm \bpm 1 & 0 \\ cd^{-1} & 1\ebpm \right)\\
	& \\
	&\hskip 50pt =\begin{cases} A\Big(\mathfrak a, m_\mathfrak a|d^{-2}(ad-bc)y|_p^{-1}-\mu_\mathfrak a\Big)e_\infty\Big(|d^{-2}(ad-bc)y|_p^{-1}j\Big)\widetilde\chi_{_{\text{\rm idelic}}}(k_0),\\
	\;\;\;\;\;\text{ if } m_\mathfrak a|d^{-2}(ad-bc)y|_p^{-1}-\mu_\mathfrak a\in\mathbb Z,\\
	0, \;\;\;\;\;\;\;\;\;\; \text{ otherwise.}\end{cases}\end{aligned}$$
	
Therefore, if $c^{-1}d\notin\mathbb Z_p$ then
	$$\begin{aligned} & \pi_p'\left(\bpm a & b\\ c& d\ebpm\right)\, . \,\varphi_p(y)\\
	& \\
	&\hskip 10pt =\begin{cases}\chi_p(c)e_p\left(ac^{-1}y\right)A\Big(\mathfrak a, m_\mathfrak a|d^{-2}(ad-bc)y|_p^{-1}-\mu_\mathfrak a\Big)e_\infty\Big(|d^{-2}(ad-bc)y|_p^{-1}j\Big)\widetilde\chi_{_{\text{\rm idelic}}}(k_0),\\
	\;\;\;\;\;\text{ if } m_\mathfrak a|d^{-2}(ad-bc)y|_p^{-1}-\mu_\mathfrak a\in\mathbb Z,\\
	0, \;\;\;\;\;\;\;\;\;\; \text{ otherwise.}\end{cases}\end{aligned}$$
\vskip 5pt

For fixed $\left(\begin{smallmatrix} a & b\\ c& d\esm\in GL(2, \mathbb Q_p)$ with $c\neq 0$, the function $\pi'_p\left(\left(\begin{smallmatrix} a & b\\ c & d\esm\right).\varphi_p\in S(\mathbb Q_p^\times)$ if and only if there exists an integer $M\ge0$ such that $\pi'_p\left(\left(\begin{smallmatrix} a & b\\ c & d\esm\right).\varphi_p(y)=0$ for $y\in p^M \mathbb Z_p$. If $\mu_\mathfrak a\neq 0$ then $A\left(\mathfrak a, m_\mathfrak a|y|_p^{-1}-\mu_\mathfrak a\right)\in S(\mathbb Q_p^\times)$ already. Therefore, for each cusp $\mathfrak a\in \mathbb Q\cup\{\infty\}$ with $\mu_\mathfrak a=0$, if there exists a non-negative integer $M_\mathfrak a$  such that $A(\mathfrak a, m_\mathfrak ap^m)=0$ for any integer $m\geq M_\mathfrak a$, then by the above computations, $\pi'_p\left(\left(\begin{smallmatrix} a & b\\ c& d\esm\right).\varphi_p\in S(\mathbb Q_p^\times)$ for any $\left(\begin{smallmatrix} a & b\\ c & d\esm\in GL(2, \mathbb Q_p)$. 
\vskip 5pt

Now assume that there exists a cusp $\mathfrak a\in \mathbb Q\cup\{\infty\}$ with $\mu_\mathfrak a=0$ such that for any non-negative integer $M$, there exists an integer $m\geq M$ and $A(\mathfrak a, m_\mathfrak ap^m)\neq 0$. By Theorem~\ref{t:3.8},
	$$A(\mathfrak a, m_\mathfrak ap^m) =\prod_{q\mid N} W_{f, q}\left(\bpm p^m & 0 \\ 0 & 1\ebpm \gamma_\mathfrak a^{-1}\right)\neq 0.$$	
	So for any non-negative integer $M$ there exists an integer $m\geq M$ such that  $W_{f, p}\left(\left(\begin{smallmatrix} p^m & 0 \\0 & 1\esm \gamma_\mathfrak a^{-1}\right)=\left(\pi'_p\left(\gamma_\mathfrak a^{-1}\right)\, . \,\varphi_p\right)(p^m)\neq 0$. 
Therefore
	$$\pi'_p\left(\gamma_\mathfrak a^{-1}\right) \, . \,\varphi_p(y)\;\notin \; S(\mathbb Q_p^\times).$$
Then $(\pi_p, V_p)$ is not supercuspidal. 
\end{proof}

\vskip10pt\noindent
\begin{corollary}\label{c:3.10}
Let $f$ be a normalized Maass-Hecke newform of weight $k$, type $\nu$, level $N$,  and  character $\chi\pmod{N}$ where $\chi$ is a primitive Dirichlet character.  Let $f_{_{\text{adelic}}}$ be the adelic lift of $f$  as in \eqref{e:AdelicLift}.

For a prime $p$, let $(\pi_p, V_p)$ be an irreducible and admissible representation of $GL(2, \mathbb Q_p)$.  If $(\pi_p, V_p)$ is isomorphic to the local component of the irreducible global automorphic representation of $GL(2, \mathbb A)$ (which is generated by $f_{_{\text{adelic}}}$), then $(\pi_p, V_p)$ cannot be supercuspidal. 
\end{corollary}
\begin{proof}
If $p\mid N$ and $\chi$ is primitive, we know that $A(\infty, p)\neq 0$ by \cite{Iwaniec:1997}. Since $f$ is a Maass-Hecke newform, for any positive integer $m$, $A(\infty, p^m) = A(\infty, p)^m\neq 0.$ By Theorem~\ref{t:3.9}, $(\pi_p, V_p)$ cannot be supercuspidal. 
\end{proof}
\vskip 10pt
\noindent{\bf Remark:}
In the proof of Theorem~\ref{t:3.9} we say that  if $(\pi_p, V_p)$ is supercuspidal then $p | N.$
In fact, more it true:  it follows directly from \cite{Casselman:1973}
that if $(\pi_p, V_p)$ is supercuspidal then $p^2 | N.$
For the convenience of the reader we briefly show how to deduce this fact from \cite{Casselman:1973}. We shall assume the reader is familiar with the notation of \cite{Casselman:1973} for the remainder of this paragraph.  Suppose that the field $k$ considered in \cite{Casselman:1973} is $\mathbb  Q_p$ and the representation $\varrho$ considered in Theorem 1 of \cite{Casselman:1973} comes from a Maass form of level $N$, such that $p^\alpha | N$ and $p^{\alpha+1} \nmid N.$    Then Casselman's conductor $c(\varrho)$ is the ideal $p^{\alpha}\cdot \mathbb Z_p.$ It is also shown in \cite{Casselman:1973} (page 304, line 16) that in the supercuspidal case  $c(\varrho) = p^{-n_1}$
where $n_1$
is a certain integral invariant of the representation $\varrho$ which was shown by 
Jacquet-Langlands 
to be at most $-2.$ 
(See \cite{Casselman:1973}, p. 303, paragraphs 1 and 2.)
Thus $\alpha = -n_1 \ge 2.$

\section{Toward a classical proof of Theorem~\ref{t:3.8}}\label{s:TowardsClassicalProof}
In this section we study the problem of giving a proof of Theorem~\ref{t:3.8} in purely classical 
terms, with a partial result in this direction being given in Theorem
~\ref{t:ThreeTerm}
below.  It is helpful to 
first review the classical proof of multiplicativity of Fourier coefficients at infinity.
Suppose that 
\begin{equation}\label{e:HeckeEigen}
\lambda \cdot f(z) = \left(T^\chi_p f\right)(z):=\frac{1}{\sqrt{p}}\left(\chi(p)f(pz) + 
\sum_{b=0}^{p-1} f\left(\frac{z+i}{p}\right)
\right),
\end{equation}
for some $\lambda \in \mathbb C$ and  all $z \in \mathfrak h.$  Suppose further that
\begin{equation}\label{e:FourExp}
f( z) = \sum_{n \ne 0}
A(\infty,n) \, W_{\frac{\text{\rm sgn}(n) k}{2}, \; \nu-\frac12} \Big( 4\pi |n|\cdot y  \Big)e^{2\pi inx}.
\end{equation}
Now plug \eqref{e:FourExp} into \eqref{e:HeckeEigen}, using the identity 
$$\sum_{b = 0}^{p-1} e^\frac{2 \pi i nb}{p} = \begin{cases} p & \text{if } p|n, \\ 0& \text{if } p\hskip -2pt\not| \, n.\end{cases}$$
  By comparing coefficients of 
  $W_{\frac{\text{\rm sgn}(n) k}{2}, \; \nu-\frac12} \Big( 4\pi |n|\cdot y  \Big)e^{2\pi inx}$
  in \eqref{e:HeckeEigen}, we see that 
  \begin{equation}\label{e:ThreeTerm}
  \frac{\chi(p)}{\sqrt{p}}A\left(\infty, \frac np\right )
  -\lambda A(\infty, n)+
   \sqrt{p} A(\infty, np) \; = \; 0,
   \end{equation}
with the understanding that $A\left(\infty, \frac np\right)=0$ if $p\hskip -3pt \not|n.$
It follows that for all $n$ with $\text{gcd}(n,p) = 1,$ and all $k \ge 0,$ 
 the coefficient $A(\infty, np^k) = b_k \cdot A(\infty, n),$ where $(b_k)_{k=-1}^\infty$
 is the unique sequence satisfying the recurrence relation 
 $$ \sqrt{p} b_k - \lambda b_{k-1} + \frac{\chi(p)}{\sqrt{p}} b_{k-2} = 0,$$
 and the initial conditions $b_{-1}=0, \; b_0 =1.$  

It is natural to ask whether one may prove that the Fourier coefficients at cusps 
other than $\infty$ satisfy a recurrence relation analogous to \eqref{e:ThreeTerm}.  The answer, in 
general, appears to be ``no.''  In fact, what we shall prove in Theorem~\ref{t:ThreeTerm} below is a 
formula which, in general, involves Fourier coefficients at {\it three different cusps.}

\begin{theorem}\label{t:ThreeTerm}
Let $f$ be a Maass form of weight $k$ level $N$ and character $\chi.$  
Assume that a maximal set of $\Gamma_0(N)$-inequivalent 
cusps and a matrix $\gamma_\fa$ for each cusp $\fa$ in this set have 
been chosen as in 
section~\ref{s:RemarksOnChoices}.  Let $A(\fa, n)$ denote the Fourier 
coefficients defined using the matrices $\sigma_\fa$ determined
by these choices of $\fa, \gamma_\fa.$ 
Let $p$ be a prime which does not divide $N,$ and assume that 
$T_p^\chi f = \lambda f.$  
Then for each representative cusp $\fa$ there exist 
cusps 
$\fa'$ and $\fa'',$ and integers 
$j', j'', \lambda', \lambda'',$ such that 
$$\begin{aligned}  &
\frac{\chi(\lambda'')}{\sqrt{p}}A\Big(\fa'', \frac{n+\mu_\fa}{p} - \mu_{\fa''}\Big)e(j'' ( n + \mu_\fa) ) \;
   - \; \lambda A(\fa ,  n)\\
   &\hskip 100pt + \;
\sqrt{p}   \chi(\lambda') e\left( j' \frac{n+\mu_\fa} p\right)A( \fa', \lfloor 
 p \cdot ( n + \mu_\fa)\rfloor ) \; = \; 0,\end{aligned}
$$
with the caveat that $\frac{n+\mu_\fa}{p} - \mu_{\fa''}$ need not be integral, and 
the convention that if it is not then $A\left(\fa'', \;\frac{n+\mu_\fa}{p} - \mu_{\fa''}\right)$
is defined to be zero.
The cusps $\fa',$ and $\fa''$ are $p\cdot \fa$ and $p^{-1}\cdot \fa$, 
respectively, in terms of  the action of $\Z/N\Z$ on a set of 
cusp representatives satisfying our hypotheses which was defined at the 
end of the last section.  

For prime power $q^e,$ the values of $\lambda', \lambda'',$ are given  in the 
table below. 
The integers $j', j'',$ are both zero unless $\fa = \frac 1{c_1q^l}$ with 
$l < \frac e2.$ 
In that case, we have
$$\begin{aligned} 
0 \le j', j'' < q^{e-l}, \quad&\quad 0 \le c_1', c_1''< q^{l},\\
c_1''p \equiv c_1 \pmod{q^{l}},\quad&\quad c_1' \equiv c_1p \pmod{q^{l}},\\
 c_1''( p - j''q^l c_1) \equiv c_1 \pmod q^e,\quad&\quad (1-q^l c_1 j' p) c_1' \equiv c_1p \pmod{q^e}.
 \end{aligned}
 $$
 For $N = \prod_{q\in S} q^{e_q},$ the integer  $\lambda'$ may be taken to be the smallest
 positive solution to the system of congruences: $\lambda' \equiv \lambda'_q \pmod{q^{e_q}}, \; (\forall \; q \in S),$
 where, for each $q,$ the integer $\lambda'_q$ is obtained by applying the result to the prime power $q^{e_q}$ 
 and the cusp $\fa_q$ for $\Gamma_0(q^{e_q})$ such that $\gamma_\fa \equiv \gamma_{\fa_q} 
 \pmod{q^{e_q}}.$  The values of $j',j'', \lambda''$ are obtained similarly.
\end{theorem}
$$
\begin{matrix}
\underline{\phantom{xxx}\fa\phantom{xxx}} &\m\m &\underline{\phantom{xxx}\lambda'\phantom{xxx}} &\m\m& \underline{\phantom{xxx}\lambda''\phantom{xxx}}\\
\infty &&1&&1\\
 0&& 1 && p \\
\frac 1{c_1 q^l}, \; l \ge \frac e2 && p&&1 \\
\frac 1{c_1 q^l}, \; l < \frac e2  && (c_1')^{-1} c_1 && (c_1'')^{-1} c_1 
\end{matrix}$$
(Inverses taken modulo $q^e.$)
\vskip 10pt 
\noindent{\bf Remarks:}
If $\chi$ is not primitive, it may not be necessary for $\lambda'$ to be congruent to $\lambda'_q$ 
modulo $q^{e_q}$;  a lower power of $q$ may suffice.  The same applies to $\lambda''.$  It 
will turn out that the systems of congruences for $j'$ and $j''$ can always be taken modulo  properly 
lower powers of $q.$   (See $e_q'$ and $N'$ in the statement of Proposition~\ref{p:7.5}.)  
\vskip 10pt 
Theorem~\ref{t:ThreeTerm} follows easily  from the following proposition, by the 
same argument sketched in the classical case at the beginning of 
this section.  This proposition works with the group algebra $\ga.$ 
The action of $\gltqp$ on functions $\uhp \to \C$ by the weight $k$ slash operator $|_k$ 
extends by $\C$-linearity to an action of $\ga,$ and this identifies the 
Hecke operator $T_p^\chi$ with 
$$\chi(p) \begin{pmatrix} p&0\\0&1 \end{pmatrix} + \sum_{b=0}^{p-1} \begin{pmatrix} 1&b\\0&p \end{pmatrix} 
\; \in \; \ga.$$

We shall also consider the right ideal in $\ga$ generated by all $\gamma - \chi(\gamma)$
with $\Gamma \in \Gamma_0(N).$  It is easy to see that any modular form of character $\chi$ for $\Gamma_0(N)$
is annihilated by this ideal.

\begin{proposition}\label{p:7.5}
Fix a cusp $\fa$ and let $\fa', \fa'', \lambda', \lambda'', j', j''$ be defined as in Theorem~\ref{t:ThreeTerm}.
For each $q \in S,$ let 
$$e_q' = \begin{cases} 0,& \fa_q = \infty, \\
e_q, & \fa_q = 0,\\
e_q - l, & \fa_q = \frac1{cq^l}.
\end{cases}$$
Let $N' = \prod_{q\in S} q^{e_q'}.$
Let $\mathcal I$ denote the right ideal in the group ring 
$\ga$ generated by all $\gamma - \chi(\gamma),\; \gamma \in \Gamma_0(N).$  Then 
$$T_p^\chi\gamma_\fa \; \equiv \; \chi(\lambda'') \gamma_{\fa''} \bpm p&0\\0&1 \ebpm 
\bpm 1& j'' \\ 0& 1 \ebpm +
\chi( \lambda' ) \gamma_{\fa'} 
\underset{b \equiv j' \pmod{N'}} 
{\sum_{
0 \le b < N'p
}
}
\bpm 1&b\\0&p\ebpm
\pmod{\mathcal I}
.$$
\end{proposition}
\vskip 10pt 
As a first step towards the proof of Proposition~\ref{p:7.5}, we show the following.
\begin{lemma}\label{l:7.6}
For each element of 
the set  $\left\{ (c, d) \in (\Z/N\Z)^2 \;|\; \langle c, d \rangle  = \Z/N\Z\right\},$  
fix an element $s(c,d)$ of $SL(2, \Z)$ such that the bottom row of $s(c,d)$
is congruent to $(c,d)\pmod{N}.$ 
Take $p$ a prime not dividing $N.$  
Then for any matrix $\gamma\in SL(2, \Z)$ such that the bottom 
row of $\gamma$ is congruent to $(c,d)\pmod{N},$ we have 
$$T_p^\chi \gamma
\equiv s(c,pd) \begin{pmatrix} p&0\\0&1 \end{pmatrix} + \sum_{i=0}^{p-1} s(pc, d-ic) \begin{pmatrix} 1&i \\ 0 & p 
\end{pmatrix} \pmod{\mathcal I},$$
where $\mathcal I$ is the right ideal of $\ga$ defined in Proposition~\ref{p:7.5}.
\end{lemma}
\begin{proof}
Let 
$$S_p = \left\{ \bpm p&0 \\ 0 &1 \ebpm \right\} \cup \left\{ \bpm 1 &i \\ 0 &p \ebpm : 0 \le i \le p \right\}.$$
Then, for all $\xi \in S_p,$
$$SL(2, \mathbb Z) \cdot \xi \cdot SL(2, \Z) = \coprod _{\xi' \in S_p} SL(2, \Z) \cdot \xi'
= \coprod _{\xi' \in S_p} \xi'\cdot SL(2, \Z).
$$
It follows that for all 
$\xi \in S_p, \gamma \in SL(2,\Z)$ there exist unique 
$\xi' \in S_p, \gamma' \in SL(2, \Z)$ such that 
$\xi \gamma = \gamma' \xi',$ and that, for fixed $\gamma,$ the map $\xi \mapsto \xi'$ is a bijection 
$S_p \to S_p.$  
If $\xi$ and $\gamma$ are given, then 
the element $\xi'$ may be described concretely as the unique element 
of $S_p$ such that $\xi \gamma (\xi')^{-1} \in SL(2, \Z).$

Fix any $\gamma_0 \in \Gamma_0(N)$ such that $\widetilde \chi(\gamma_0) = \chi(p),$
and let 
$$S'_p =  \left\{ \gamma_0 
\bpm p&0 \\ 0 &1 \ebpm \right\} \cup \left\{ \bpm 1 &i \\ 0 &p \ebpm : 0 \le i \le p \right\}.$$
  Then 
  $$T_p^\chi 
  = \sum_{\xi \in S'_p} \xi  \pmod{\mathcal I}.$$
  To prove that two individual matrices in $\gltqp$ are equivalent $\pmod{\mathcal I}$ it 
  suffices to prove that their bottom rows 
  are congruent modulo $N.$  
  It is clear that the bottom row of $\xi s(c,d)$ is $(pc, pd)$ for all $\xi \in S'_p,$
  so that the bottom row of $\xi s(c,d) (\xi')^{-1}$ is the row vector
  $(pc, pd) \cdot (\xi')^{-1}.$   As $\xi$ ranges over $S'_p,$ this row vector 
  ranges over the set 
  $\{ (pc, d-ic) \} \cup \{ (c, pd)\}.$  This completes the proof.
\end{proof}
\vskip 10pt 
\noindent
{\bf Proof of Proposition~\ref{p:7.5}:}
We assume that we have fixed a maximal set of $\Gamma_0(N)$-inequivalent 
cusps together with a choice of matrix $\gamma_\fb$ for each cusp $\fb$
in this set as in 
section~\ref{s:RemarksOnChoices}.
Let $$X = \left\{ (c, d) \in (\Z/N\Z)^2 \;|\; \langle c, d \rangle  = \Z/N\Z\right\}.$$  For 
any $x= (x_0,x_1) \in X,$ there exist a unique cusp $\fb$ from our fixed set of representatives, 
corresponding to the $\Gamma_\i$-orbit of $[x_0:x_1]$ in $\P^1( \Z/N\Z).$ 
If $\gamma \in SL(2, \Z)$ has bottom row $\equiv (x_0,x_1) \pmod{N},$ then it is 
possible to choose $\gamma_0 \in \Gamma_0(N)$ and $j \in \Z$ such that 
$\gamma = \gamma_0 \gamma_\fb \left(\begin{smallmatrix} 1&j \\ 0 &1 \esm.$  
It follows immediately that 
$\gamma \equiv \widetilde\chi(\gamma_0)  \gamma_\fb \left(\begin{smallmatrix} 1&j \\ 0 &1 \esm \modi.$ 
It is possible to make 
$\gamma_0$ and $j$ unique by requiring $0 \le j \le m_\fb,$ but we will not do so.  
It turns out that it is sometimes 
possible to gain a bit of control over $\widetilde \chi( \gamma_0)$ by 
allowing 
 $j$ to vary over a larger range, 
 and this approach is more convenient.

In order to prove Proposition~\ref{p:7.5}, we must carry out this analysis carefully 
for each element of the set $\{ (pc, d-ic ) \} \cup \{ (c , pd)\},$
as $(c ,d)$ ranges over the bottom rows of the matrices $\gamma_\fa.$
We first complete the case when $N = q^e.$  
In this case, the pairs $(c ,d)$ considered  are
$ (1,0),(0,1),$ and $(cq^l , 1)$ for $0 < l < e, 0 < c < q^{e-l}$ and $q\nmid c.$
Let us first look carefully at $(pc , d-ic )$ in the case when $(c , d) = (1,0).$ 
clearly $(pc , d-ic ) = (p, -i)$ in this case.  The corresponding element 
of $\P^1( \Z/ q^e\Z)$ is $[p: -i],$ or $[1: -i \overline{p}],$ where $\overline{p}$ 
denotes the inverse modulo $q^e.$  This element of $\P^1( \Z/q^e\Z)$ lies in the 
$\Gamma_\infty$-orbit for which the standard representative is $[1:0].$  
This tells us  that $\gamma_{\fb}$ in this case is $\left(\begin{smallmatrix} 0&-1\\1&0 \esm.$  
Furthermore, since 
$[1:-i \overline{p} ] = [1:0]\cdot \left(\begin{smallmatrix} 1&-i \cdot \overline{p}\\ 0&1 \esm,$
we see that $j$ may be taken to be $-i \cdot \overline{p}.$  
(Here, we interpret this as taking the additive inverse and product modulo 
$q^e,$ so that the answer is an integer between $0$ and $q^e -1,$ inclusive.)

In order to 
keep track of $\widetilde \chi( \gamma_0)$ as well, it is necessary to work with 
$X$ rather than $\P^1.$  The more precise statement is that 
$(p, -i) = p \cdot (1,0) \cdot \left(\begin{smallmatrix} 1& -i \overline{p} \\ 0&1\esm.$
This immediately implies that 
$$\begin{pmatrix} *&*\\ p&-i\end{pmatrix} = \gamma_0 \cdot \begin{pmatrix} 0&-1\\ 1&0 \end{pmatrix} \cdot \begin{pmatrix} 1& -i \overline{p} \\ 0&1\end{pmatrix},$$
with $\gamma_0 \in \Gamma_0(q^e)$ such that 
$\widetilde \chi(\gamma_0) = \chi( p).$ 
Hence 
$$\begin{pmatrix} *&*\\ p&-i\end{pmatrix} \; \equiv \; \chi(p) \cdot \begin{pmatrix} 0&-1\\ 1&0 \end{pmatrix} \cdot \begin{pmatrix} 1& -i \overline{p} \\ 0&1\end{pmatrix}
\modi$$
in this case.
Now, $\left(\begin{smallmatrix}  1& -i \overline{p} \\ 0&1\esm \cdot \left(\begin{smallmatrix} 1&i \\ 0 &p\esm 
= \left(\begin{smallmatrix} 1& i - i p\overline{p} \\ 0& p \esm.$
Clearly, for each $i,$ the integer $i-i p\overline{p}$ is divisible by $q^e.$  Furthermore, these 
integers are all distinct modulo $p,$ and all in the range from $0$ to $pq^e-1.$  It follows 
that we have 
$$\sum_{i = 0}^{p-1} 
  \begin{pmatrix} *&*\\ p&-i\end{pmatrix} \cdot \begin{pmatrix} 1&i \\ 0 &p\end{pmatrix} 
 \; \equiv \; \chi(p) \cdot \begin{pmatrix} 0&-1\\ 1&0 \end{pmatrix} \cdot \underset{ b \equiv 0 \pmod{q^e} } 
  {\sum_{0 \le b \le q^ep}}\begin{pmatrix} 1&b\\0&p\end{pmatrix}.$$

Having worked this case in detail to illustrate the method, we shall henceforth be more 
brief.  To finish the case $(1,0)$ we simply note that if $(c,d)=(1,0)$ then 
$(c, pd)$ is again $(1,0),$ whence a contribution of 
$\left(\begin{smallmatrix} *&*\\  1& 0 \esm \equiv \left(\begin{smallmatrix} 0& -1 \\ 1&0 \esm \modi.$
This completes the proof of Proposition~\ref{p:7.5} in the case $N = q^e$ and $\fa = 0.$
The case $\fa = \infty$ is, of course,  trivial.  

We turn to $\fa = \frac 1{c_1 q^l}.$  
First, consider.
$(c_1q^l , p )$.  
We must identify the standard representative for the orbit of this element under 
the action of $\Z/N\Z \times \Gamma_\i,$ where $\Z/N\Z$ acts by scalar multiplication 
and $\Gamma_\i$ acts by matrix multiplication on the right.  Referring to 
Lemma~\ref{l:4.1}, it is 
not difficult to see that the standard representative is $(c_1''q^l , 1)$ where 
$c_1''$ is uniquely determined by the conditions that $pc_1'' \equiv c_1 \pmod{q^{e-l}}$ and 
$0< c_1'' < q^{\min(l,e-l)}.$ 
In the simpler case $l \ge e-l,$ we get
$$\begin{pmatrix} * & * \\ c_1q^l & p \end{pmatrix} \; \equiv \;
\chi(p) \begin{pmatrix} 1& 0 \\ c_1'' q^l& 1\end{pmatrix} \modi.$$

If $l < e-l,$ it is a bit more complicated.  
Because we deal with reduction and inversion modulo various powers
of $q,$ it is convenient to introduce a more explicit notation.  
For $a,r$ integers with $r$ positive, 
let $[a]_r$ denote the unique integer between $0$ and $q^r-1,$ inclusive,
which is congruent to $a\pmod{q^r}$ and $[a]_r^{-1}$ 
 the unique integer in the same range such that 
 $a\cdot [a]_r^{-1} \equiv 1 \pmod{q^r}.$

When $l < e-l,$ we have $c_1'' = c_1 \big[[p]_{e-l}^{-1}\big]_l.$  It is necessary to choose
$j''$ such that $c_1 [p - j''c_1q^l]^{-1}_{e-l} \equiv c_1'' \pmod{q^{e-l}}.$ 
However, it is possible, and more convenient, to choose $j''$ subject to the 
more stringent condition that 
$c_1 [p - j''c_1q^l]^{-1}_{e} \equiv c_1'' \pmod{q^{e}}.$ 
This ensures that $\chi(p - j''c_1q^l) = \chi(c_1)\chi(c_1'')^{-1}.$
We have 
$$( c_1q^l, p) = (c_1q^l, p-j''c_1q^l) \begin{pmatrix} 1&j''\\0&1 \end{pmatrix}
\; = \; (p -j''c_1q^l) ( c_1''q^l, 1)  \begin{pmatrix} 1&j''\\0&1 \end{pmatrix},$$ and deduce
$$\begin{pmatrix} *&*\\ c_1q^l&p\end{pmatrix} 
\equiv \chi(c_1)\chi(c_1'')^{-1} \begin{pmatrix} 1&0\\c_1''q^l&1\end{pmatrix} \begin{pmatrix} 1&j''\\0&1 \end{pmatrix}.$$

Now, consider $(c_1pq^l , 1-ic_1q^l).$
When $l\ge e-l$ the most expedient thing is to write this as 
$$(c_1pq^l , 1)\begin{pmatrix} 1 & [-[p]^{-1}_{e-l}\cdot i]_{e-l}\\0&1\end{pmatrix}.$$ Much as above, we get
$$\sum_{i=0}^{p-1} 
\begin{pmatrix} * & * \\ c_1pq^l & 1-ic_1q^l \end{pmatrix} \; \equiv \; \begin{pmatrix} 1&0\\c_1'q^l&1\end{pmatrix}
\underset{b \equiv 0 \pmod{q^{e-l}}} {\sum_{0\le b < pq^{e-l}}} 
\begin{pmatrix} 1& b \\ 0 & p \end{pmatrix}.$$
When $l < e-l$ we choose $j(i)$ so that $c_1p [1-ic_1q^l - j(i)c_1pq^l ]^{-1}_e \equiv c_1' \pmod{q^e}.$
It is easy to see that this has a unique solution modulo $q^{e-l},$ and that the quantity 
$i+j(i)p$ is constant, modulo $q^{e-l},$ as $i$ varies.  We get
$$\sum_{i=0}^{p-1} 
\begin{pmatrix} * & * \\ c_1pq^l & 1-ic_1q^l \end{pmatrix} \; \equiv  
\underset{b \equiv j(0) \pmod{q^{e-l}}} {\sum_{0\le b < pq^{e-l}}} 
\begin{pmatrix} 1& b \\ 0 & p \end{pmatrix}.$$  It is easily verified that $j(0)=j',$ defined as in the statement of the theorem.
This completes the proof in the prime power case.  

To complete the general case, let us resume the notation $N = \prod_{q \in S} q^{e_q}.$ 
Take $(c,d)$ the bottom row of the matrix $\gamma_\fa$ corresponding 
to some cusp $\fa,$ and consider $(c, pd).$  For each $q,$ we know that 
$(c, pd)$ is equivalent modulo $q^{e_q}$ to one of the row vectors for which the 
analysis was carried out above.  Thus, for each $q$ we obtain 
a cusp $\fa_q'',$ 
$j''_q \in \Z$ and 
$\lambda_q'' \in (\Z/q^{e_q}\Z)^\times$ such that
$ (c, pd) \equiv \lambda_q'' 
\cdot(c_q' , d_q') \cdot \left(\begin{smallmatrix} 1&j'_q\\0&1 \esm \pmod{q^{e_q}},$ where
$(c'_q,d'_q)$ is the bottom row of $\gamma_{\fa_q''}.$  
Solving systems of congruences, we obtain a cusp $\fa'',$ integer $j''$ and 
element $\lambda''$ of $(\Z/N\Z)^\times$ such that 
$$(c, pd) \; = \; \lambda'' ( c', d') \begin{pmatrix} 1&j''\\0&1 \end{pmatrix},$$
where $(c',d')$ is the bottom row of $\gamma_{\fa''}.$
It follows that 
$$\begin{pmatrix} * & * \\  c&pd \end{pmatrix} \; \equiv \; \chi( \lambda'') \gamma_{\fa''} \begin{pmatrix} 1&j''\\0&1 \end{pmatrix}\modi.$$

Similarly, for $i=0$ to $p-1,$ by passing to the individual primes, using the 
prime power case, and then solving systems of congruences, we obtain 
$$\begin{pmatrix} * & * \\  pc'&d-ipc' \end{pmatrix} \; \equiv \; \chi( \lambda') \gamma_{\fa'} \begin{pmatrix} 1&j(i)\\0&1 \end{pmatrix} \hskip -3pt \modi,
\qquad (0 \le i < p)
$$
for some cusp $\fa'$ some $\lambda' \in (\Z/N\Z)^\times,$ which are independent 
of $i,$ and some integers $j(i), 0 \le i < p.$  
  Furthermore, it follows from the 
analysis done in the prime power case above that 
$j(i) \cdot p + i \equiv j'_q \pmod{q^{e'_q}},$  for each $q,$ 
where $e'_q$ is defined as in the proposition and $j'_q$ is as described (for prime powers) 
in the statement of Theorem~\ref{t:ThreeTerm}.  
It follows at once that 
$j(i) \cdot p + i$ is constant $\pmod{N'},$  where $N'$ is defined 
as in the statement of Proposition~\ref{p:7.5}, and that the value is obtained by solving the system 
of congruences obtained from the various $q.$  
\qed


\end{document}